\documentclass[11pt]{article}
\usepackage[margin=1in]{geometry}

\usepackage{authblk} % For author info
\usepackage{graphicx}
\usepackage{amssymb}

\usepackage{amsmath}
\usepackage{enumerate}% http://ctan.org/pkg/enumerate
\usepackage{amsfonts}
\usepackage[mathcal]{euscript}
\usepackage{amsthm}
\usepackage[all,2cell,import]{xy} \UseAllTwocells
\usepackage{xcolor}
\usepackage{tkz-euclide}
\usepackage{tikz}

\usepackage{url}
\makeatletter
\def\url@leostyle{%
	\@ifundefined{selectfont}{\def\UrlFont{\sf}}{\def\UrlFont{\small\ttfamily}}}
\makeatother
\usepackage[refpage]{nomencl}

\makenomenclature
\usepackage{makeidx}
\makeindex
\numberwithin{equation}{section}

\usepackage[colorlinks=true, pdfstartview=FitV, linkcolor=black,
			   citecolor=black, urlcolor=black]{hyperref}
\usepackage{bookmark}
\usepackage[titletoc]{appendix} %For some reason, putting the bookmark and appendix packages after hyperref gets rid of the \endcsname error.
\hypersetup{hypertexnames=false} % This prevents the "same identifier" warning when compiling.

          \def\dudt{\frac{\del u}{dt}}

          \newcommand{\nc}{\newcommand}
          \nc{\DMO}{\DeclareMathOperator}	
          \nc{\commentout}[1]{}

          \nc{\newnotation}{\nomenclature}
          \nc{\wrap}{\cW}
          \nc{\Tw}{\mathsf{Tw}}
          \nc{\loc}{\mathsf{Loc}}
          \nc{\Top}{Top}
          \nc{\emb}{\mathsf{emb}}
          \nc{\ind}{\mathsf{Ind}}
          \nc{\Ind}{\mathsf{Ind}}
          \nc{\Loc}{\mathsf{Loc}}
          \nc{\Cob}{\mathsf{Cob}}
          \nc{\mul}{\mathsf{Mul}}
          \nc{\fat}{\mathsf{fat}}
          \nc{\cob}{\mathsf{Cob}}
          \nc{\coh}{\mathsf{Coh}}
          \nc{\Liouaut}{\Aut}
          \nc{\Liouauto}{{\Aut^o}}
          \nc{\Liouautb}{\Aut^{b}}
          \nc{\Liouautgr}{\Aut^{gr}}
          \nc{\Liouautgrb}{\Aut^{gr,b}}
          \nc{\idem}{\mathsf{Idem}}
          \nc{\sets}{\mathsf{Sets}}
          \nc{\near}{\mathsf{near}}
          \nc{\sing}{\mathsf{Sing}}
          \nc{\Sing}{\mathsf{Sing}}
          \nc{\perf}{\mathsf{Perf}}
          \nc{\block}{\mathsf{block}}
          \nc{\ssets}{\mathsf{sSets}}
          \nc{\cmpct}{\mathsf{cmpct}}
          \nc{\compact}{\mathsf{cmpct}}
          \nc{\pwrap}{\mathsf{PWrap}}
          \nc{\coder}{\mathsf{Coder}}
          \nc{\bimod}{\mathsf{Bimod}}
          \nc{\grmod}{\mathsf{GrMod}}
          \nc{\Morita}{\mathsf{Morita}}
          \nc{\morita}{\mathsf{Morita}}
          \nc{\spaces}{\mathsf{Spaces}}
          \nc{\pwrms}{\mathsf{PWrFuk}_{M,S}}
          \nc{\pwrmf}{\mathsf{PWrFuk}_{M,F}}
          \nc{\pwrapmf}{\mathsf{PWrFuk}_{M,F}}
          \nc{\fuk}{\mathsf{Fukaya}}
          \nc{\infwr}{\mathsf{InfWr}}
          \nc{\fukaya}{\mathsf{Fukaya}}
          \nc{\autml}{\mathsf{Aut}_{M,\Lambda}}
          \nc{\fukml}{\mathsf{Fukaya}_{M,\Lambda}}
          \nc{\fukmle}{\mathsf{Fukaya}_{M,\Lambda,\epsilon}}
          \nc{\fukmod}{\wrfukcompact(M)\modules}
          \nc{\lag}{\mathsf{Lag}}
          \nc{\lagm}{\lag_M}
          \nc{\lago}{\lag^o}
          \nc{\lagml}{\lag_{M,\Lambda}} % For when I get lazy.
          \nc{\lagmle}{\lag_{M,\Lambda,\epsilon}}
          \nc{\Fun}{\mathsf{Fun}}
          \nc{\fun}{\mathsf{Fun}}
          \nc{\vect}{\mathsf{Vect}}
          \nc{\chain}{\mathsf{Chain}}
          \nc{\chainn}{Chain}
          \nc{\wrfuk}{\mathsf{WrFukaya}}
          \nc{\wrfukcompact}{\mathsf{WrFukaya}_{\mathsf{cmpct}}}
          \nc{\pwrfuk}{\mathsf{PWrFukaya}}
          \nc{\inffuk}{\mathsf{InfFuk}}
          \nc{\pwrfukml}{\mathsf{PWrFukaya}_{M,\Lambda}}
          \nc{\inffukml}{\mathsf{InfFuk}_{M,\Lambda}}
          \nc{\nattrans}{\mathsf{NatTrans}}
          \nc{\corres}{\mathsf{Corres}}
          \nc{\fukep}{\fukaya_\Lambda(M,\epsilon)}
          \nc{\fukepop}{\fukaya_\Lambda(M,\epsilon)^{\op}}
          \nc{\lagep}{\lag_\Lambda(M,\epsilon)}
          \DMO{\cyl}{cyl} % Cylindrical
          \nc{\dbcoh}{D^b\mathsf{Coh}}
          \nc{\corr}{\mathsf{Corr}}

          \nc{\cat}{\mathsf{Cat}}
          \nc{\Cat}{\mathsf{Cat}}
          \nc{\ainfty}{\mathsf{A}_\infty}
          \nc{\inftycat}{\mathcal{C}\!\operatorname{at}_\infty}
          \nc{\inftyCat}{\mathcal{C}\!\operatorname{at}_\infty}
          \nc{\inftyGpd}{\mathcal{G}\!\operatorname{pd}_\infty}
          \nc{\Ainftycat}{\mathcal{C}\!\operatorname{at}_{A_\infty}}
          \nc{\dgcat}{\mathcal{C}\!\operatorname{at}_{dg}}
          \nc{\ainftycat}{\mathcal{C}\!\operatorname{at}_{A_\infty}}
          \nc{\stablecat}{\mathcal{C}\!\operatorname{at}_\infty^{\Ex}}

          \DMO{\im}{im}
          \DMO{\ev}{ev}
          \DMO{\stable}{Ex}
          \DMO{\inj}{inj}
          \DMO{\fib}{fib}
          \DMO{\conf}{Conf}
          \DMO{\chains}{Chains}
          \DMO{\cochains}{Cochains}
          \DMO{\cone}{Cone}
          \DMO{\Map}{Map}
          \DMO{\ran}{Ran}
          \DMO{\rot}{Rot}
          \DMO{\leg}{Leg}
          \DMO{\imm}{imm}
          \DMO{\adj}{adj}
          \DMO{\symp}{Symp}
          \DMO{\tree}{Tree}
          \DMO{\cube}{Cube}
          \DMO{\deep}{deep}
          \DMO{\back}{back}
          \DMO{\Hoch}{Hoch}
          \DMO{\front}{front}
          \DMO{\flow}{Flow}
          \DMO{\floer}{Floer}
          \DMO{\Maps}{Maps}
          \DMO{\exact}{exact}
          \DMO{\excess}{Excess}
          \DMO{\Decomp}{Decomp}
          \DMO{\decomp}{Decomp}
          \DMO{\collar}{collar}
          \DMO{\yoneda}{Yoneda}
          \DMO{\hamspace}{Ham}
          \DMO{\sympspace}{Symp}
          \DMO{\holomaps}{Holomaps}
          \DMO{\comp}{Comp}
          \DMO{\crit}{Crit}
          \DMO{\test}{{test}}
          \DMO{\sign}{sign}
          \DMO{\topp}{top}
          \DMO{\indx}{Index}
          \DMO{\Break}{Break} % Partitions
          \DMO{\zero}{zero} %Zero
          \DMO{\ob}{Ob}
          \DMO{\gr}{Gr} % Grassmanian
          \DMO{\Gr}{Gr} % Grassmanian
          \DMO{\cl}{Cl} % Clifford Algebra
          \DMO{\grlag}{GrLag}
          \DMO{\Pin}{Pin}
          \DMO{\Graph}{Graph}
          \DMO{\pin}{Pin}
          \DMO{\gap}{Gap}
          \DMO{\Ex}{Ex}
          \DMO{\id}{id}
          \DMO{\End}{End}
          \DMO{\sym}{Sym}
          \DMO{\aut}{Aut}
          \DMO{\Aut}{Aut}
          \DMO{\haut}{hAut}
          \DMO{\hAut}{hAut}
          \DMO{\DK}{DK} %Dold-Kan
          \DMO{\poly}{poly} % Polynomial deRham forms
          \DMO{\diff}{Diff}
          \DMO{\coll}{coll}
          \DMO{\dist}{dist} %Distance function
          \DMO{\coker}{coker} %Cokernel
          \nc{\kernel}{\ker} %Kernel
          \DMO{\sspan}{span}
          \DMO{\hocolim}{hocolim}	
          \DMO{\holim}{holim}
          \DMO{\sk}{sk}

          \DMO{\ho}{ho}
          \DMO{\fin}{fin}
          \DMO{\tor}{Tor}
          \DMO{\ext}{Ext}
          \DMO{\ret}{Ret}
          \DMO{\ham}{Ham}
          \DMO{\con}{con}
          \DMO{\leaf}{leaf}
          \DMO{\supp}{supp}
          \DMO{\edge}{edge}
          \DMO{\colim}{colim}
          \DMO{\edges}{edges}
          \DMO{\Image}{image}
          \DMO{\roots}{roots}
          \DMO{\height}{height}
          \DMO{\finmod}{FinMod}
          \DMO{\leaves}{leaves}
          \DMO{\planar}{planar}
          \DMO{\vertices}{vertices}

          \nc{\lagg}{\lag^{\cG}}
          \nc{\iso}{\mathsf{Iso}}
          \nc{\Set}{\mathsf{Set}}
          \nc{\Ass}{\mathsf{ \bf Ass}}
          \nc{\Mod}{\mathsf{Mod}}
          \nc{\modules}{\mathsf{Mod}}
          \nc{\sset}{\mathsf{sSet}}
          \nc{\liou}{\mathsf{Liou}}
          \nc{\poset}{\mathsf{Poset}}
          \nc{\trno}{T^*\RR^n_{\geq 0}}
          \nc{\spectra}{\mathsf{Spectra}}
          \nc{\tensorfin}{\tensor^{\fin}}
          \nc{\lagptg}{\lag_{pt,pt}^{\cG}}
          \nc{\Fin}{\mathcal{F}\mathsf{in}}
          \nc{\lagnl}{\lag_{N,\Lambda}}
          \nc{\lagmlg}{\lag_{M,\Lambda}^{\cG}}
          \nc{\lagsplit}{\lag^{\mathsf{split}}}
          \nc{\lagktimes}{(\lag^{\dd k})^\times}
          \nc{\lagplanar}{\lag^{\times,\planar}}
          \nc{\Cont}{\text{\rm Cont}}
          \nc{\Ham}{\text{\rm Ham}}
          \nc{\Dev}{\text{\rm Dev}}
          \nc{\Lin}{\text{\rm Lin}}
          \nc{\Int}{\text{\rm Int}}
          \nc{\Hom}{\text{\rm Hom}}
          \nc{\Chord}{\text{\rm Chord}}
          \nc{\nbhd}{\mathcal{N}\text{\rm{bhd}}}

          \nc{\onef}{1_{\fukaya}}
          \nc{\smsh}{\wedge}
          \nc{\un}{\underline}
          \nc{\xto}{\xrightarrow}
          \nc{\xra}{\xto}
          \nc{\tensor}{\otimes}
          \nc{\del}{\partial}
          \nc{\dd}{\diamond}
          \nc{\tri}{\triangle}
          \nc{\bb}{\Box}
          \nc{\into}{\hookrightarrow}
          \nc{\onto}{\twoheadrightarrow}
          \nc{\contains}{\supset}
          \nc{\transverse}{\pitchfork}
          \nc{\uncirc}{\underline{\circ}}

          \nc{\Jbar}{\overline{J}}
          \nc{\Fbar}{\overline{F}}
          \nc{\delbar}{\overline{\del}}
          \nc{\thetabar}{\overline{\theta}}
          \nc{\omegabar}{\overline{\omega}}
          \nc{\Liou}{\text{\rm Liou}}

          \nc{\Yhat}{\widehat{Y}}

          \nc{\trbar}{\overline{T^*\RR}}
          \nc{\tr}{T^*\RR}
          \nc{\tsa}{Ts\cA}
          \nc{\tsb}{Ts\cB}
          \nc{\cmbar}{\overline{\cM}}
          \nc{\crbar}{\overline{\cR}}

          \nc{\vece}{ {\vec \epsilon}}	
          \nc{\vecd}{ {\vec \delta}}
          \nc{\ov}{\overline}
          \DMO{\op}{op}
          \nc{\opp}{ ^{\op}}

          \nc{\hiro}{\textcolor{blue}}
          \nc{\YG}{\textcolor{orange}}

          \nc{\eqn}{\begin{equation}}
          \nc{\eqnn}{\begin{equation}\nonumber}
          \nc{\eqnd}{\end{equation}}
          \nc{\enum}{\begin{enumerate}}
          \nc{\enumd}{\end{enumerate}}
          \nc{\beastar}{\begin{eqnarray*}}
          \nc{\eeastar}{\end{eqnarray*}}

          \def\cA{\mathcal A}\def\cB{\mathcal B}
          \def\cF{\mathcal F}\def\cG{\mathcal G}\def\cH{\mathcal H}
          \def\cJ{\mathcal J}\def\cL{\mathcal L}
          \def\cM{\mathcal M}\def\cO{\mathcal O}\def\cP{\mathcal P}
          \def\cR{\mathcal R}\def\cS{\mathcal S}
          \def\cU{\mathcal U}\def\cW{\mathcal W}

          \def\CC{\mathbb C}
          \def\HH{\mathbb H}
          \def\JJ{\mathbb J}\def\LL{\mathbb L}
          
          \def\RR{\mathbb R}
          
          \def\ZZ{\mathbb Z}

          \def\sS{\mathsf S}

          \nc{\Euc}{\mathsf{Euc}}
          \nc{\mfld}{\mathsf{Mfld}}
          \nc{\DTop}{\mathsf{DTop}}
          \nc{\simp}{\mathsf{Simp}}
          \nc{\Ainftycatt}{A_\infty Cat}
          \nc{\dgcatt}{dg Cat}
          \nc{\StableCat}{StableCat}
          \nc{\subdivision}{\mathsf{subdiv}}
          \nc{\Kan}{\mathcal{K}\mathsf{an}}

          \def\dudt{\frac{\del u}{\del t}}

          \theoremstyle{definition}
          \newtheorem{theorem}{Theorem}[section]
          
          \newtheorem{prop}[theorem]{Proposition}
          \newtheorem{lemma}[theorem]{Lemma}
          \newtheorem{warning}[theorem]{Warning}

          \newtheorem{construction}[theorem]{Construction}
          
          \newtheorem{definition}[theorem]{Definition}
          \newtheorem{defn}[theorem]{Definition}
          \newtheorem{notation}[theorem]{Notation}
          \newtheorem{example}[theorem]{Example}
          \newtheorem{choice}[theorem]{Choice}

          \newtheorem{remark}[theorem]{Remark}
          \newtheorem{figurelabel}[theorem]{Figure} % Added for indexing figures more easily.

\title{Holomorphic curves and continuation maps in Liouville bundles}

\author[$\dagger$]{Yong-Geun Oh}
\author[$\star$]{Hiro Lee Tanaka}

\affil[$\dagger$]{Center for Geometry and Physics (IBS), Pohang, Korea \&
Department of Mathematics, POSTECH, Pohang Korea.}
\affil[$\star$]{Department of Mathematics, Texas State University}

\begin{document}

\maketitle

\begin{abstract}
We construct an unwrapped Floer theory for bundles of Liouville manifolds (and Liouville sectors more generally). This means we construct a compatible collection of unwrapped Fukaya categories of fibers of a Liouville bundle, and prove that the two natural constructions of continuation maps in this setting behave compatibly. This sets up a machinery for studying Floer-theoretic invariants of smooth group actions on Liouville manifolds and sectors; indeed, these constructions are exploited in~\cite{oh-tanaka-actions} to construct homotopically coherent actions of Lie groups on wrapped Fukaya categories, thereby proving a conjecture from Teleman's 2014 ICM address.
\end{abstract}

\tableofcontents

\section{Introduction}
In this paper, we lay the groundwork for non-wrapped Floer theory in bundles of Liouville manifolds, and of Liouville sectors more generally.

A Liouville manifold is a kind of non-compact symplectic manifold\footnote{A zero-dimensional Liouville sector may be compact, however.}, equipped with the data of a well-behaved antiderivative for the symplectic form. Liouville manifolds form a well-studied class of symplectic manifolds---the non-compactness provides a flexibility similar to the flexibility enjoyed by non-compact manifolds in differential topology, while the behavior of the antiderivative guarantees the safe use of holomorphic curve techniques to study Liouville manifolds. Examples include cotangent bundles, Weinstein manifolds, and completions of Liouville domains.

A Liouville sector is a generalization introduced in~\cite{gps}. Informally, Liouville manifolds are to Liouville sectors what manifolds are to manifolds with boundary---indeed, a Liouville sector without boundary is simply a Liouville manifold, and when a Liouville sector has boundary, one demands that the characteristic foliation of the boundary satisfy some constraints that again permit the safe use of holomorphic curve techniques. For example, if $Q$ is a smooth manifold, $T^*Q$ is a Liouville manifold (and sector), while if $Q$ is a smooth manifold with non-empty boundary, $T^*Q$ is a Liouville sector.

There is a natural notion of automorphism of Liouville sector $M$, and a bundle of Liouville sectors over a smooth manifold $B$ is an $M$ fiber bundle whose structure group has been smoothly reduced to this automorphism group. An automorphism of a Liouville sector without boundary is precisely an automorphism as a Liouville manifold. Thus, for most of this paper,  the reader familiar with Liouville manifolds will not lose much intuition by having in mind the case of Liouville manifolds (rather than sectors).

To motivate the study of Floer theory of bundles, let us note that a plethora of Floer-theoretic calculations has been successfully carried out precisely by exploiting symmetries. Indeed, many of the first computations of Lagrangian Floer cohomology arose by studying fixed points of antiholomorphic involutions, or by exploiting torus actions on toric manifolds.

Where do bundles enter the picture? In homotopy theory, given a group action of $G$ on an object $Y$, it is often convenient to exhibit a family of $Y$ living over the classifying space $BG$. A standard way to do so is to construct a map $G \to \Aut(Y)$ (and hence a map $BG \to B \Aut(Y)$), exhibit some universal $Y$-bundle living over $B\Aut(Y)$, and pull back this universal bundle along the map $BG \to B\Aut(Y)$. One can perform this construction for any Liouville action of $G$ on a Liouville sector $M$, exhibiting a family of $M$ living over $BG$ (i.e., an $M$-bundle over $BG$). Moreover, by combining our present work with the diffeological space framework as exploited in~\cite{oh-tanaka-smooth-approximation}, one can articulate the {\em smoothness} of such infinite-dimensional entities.

Building on the results of~\cite{oh-tanaka-smooth-approximation}, the present work exploits this smoothness to construct a system of Fukaya categories living over $BG$. Informally, to any simplex mapping smoothly to $BG$, we assign a Fukaya category; our assignment respects the inclusions of the faces of simplices in a standard way. (Theorem~\ref{theorem. O is Aoo} below.) These Fukaya categories are unwrapped and directed, in a sense explained below.

The culmination of all this legwork is our work~\cite{oh-tanaka-actions}, where we  incorporate $\infty$-categorical machinery from~\cite{oh-tanaka-localizations} to show that the {\em wrapped} Fukaya category of $M$ varies locally constantly over $BG$. (For a brief discussion of the importance of wrapping, see Section~\ref{section. motivating continuation compatibility} below.) In other words, one obtains a local system of $A_\infty$-categories over $BG$. Again, standard homotopy-theoretic techniques tell us that such a local system precisely encodes the data of a $G$ action on the wrapped Fukaya category of $M$. As it happens, this proves a conjecture of Teleman from the 2014 ICM~\cite{teleman-icm} in the Liouville and the monotone settings.

This opens the door to exploit symmetries of Liouville manifolds, and Liouville sectors more generally.

For example, the work in~\cite{oh-tanaka-actions} shows that the homotopy groups of $G$ map to the Hochschild cohomology groups of the wrapped Fukaya category of $M$; even better, these homomorphisms arise from a map of $E_2$-algebras from the based loop space $\Omega G$ to the space of Hochschild cochains.

As another application, if one is interested in invariants up to homotopy coherent data, one can use our techniques to avoid speaking of honestly $G$-equivariant data on $M$, and instead study data encoded by certain holomorphic disks in the Liouville bundle over $BG$. This avoids some difficulties associated to constructing equivariant Floer theory. (However, we warn that the ``strictly'' $G$-equivariant Floer theory will often {\em not} be equivalent to the homotopically flavored data one obtains by constructing Floer-theoretic invariants over $BG$. There are pros and cons to both models: The homotopical constructions will almost always be better behaved, but a strict construction may contain geometric data invisible in the homotopical world.)

Let us state the main results of the present work.

Fix a Liouville bundle $E \to B$. We first define a directed Fukaya category $\cO_j$ for every simplex $j: |\Delta^n| \to B$ smoothly mapping to $B$. The {\em directedness} means that all endomorphism algebras are equivalent to the base ring, and that the collection of objects of $\cO_j$ has a partial ordering---if $L_0 \not \leq L_1$, then there are only zero morphisms from $L_0$ to $L_1$. The totality of this data (i.e., the collection of these Fukaya categories), along with their compatibilities along face maps of simplices, is what one might call the (unwrapped, directed) Floer theory associated to a Liouville bundle.

\begin{theorem}\label{theorem. O is Aoo}
For every smooth map $j: |\Delta^n| \to B$, the unwrapped, directed Fukaya category $\cO_j$ is an $A_\infty$-category.
Moreover, for every commutative diagram
    \eqnn
    \xymatrix{
    |\Delta^n| \ar[rr]^\iota \ar[dr]^j
        && |\Delta^{n'}| \ar[dl]^{j'} \\
        & B &
    }
    \eqnd
where $\iota$ is an injective simplicial map, we have an induced fully faithful functor $\iota_*: \cO_j \to \cO_{j'}$. These functors respect composition of $\iota$.

In other words, the assignment $j \mapsto \cO_j$ and $\iota \mapsto \iota_*$ defines a functor
$$
\cO: \simp(B) \to \Ainftycatt
$$
from the category $\simp(B)$ of smooth simplices of $B$ to the category $\Ainftycatt$ of $A_\infty$-categories, where all morphisms of $\simp(B)$ are sent to fully faithful maps.
\end{theorem}

Informally, an object of $\cO_j$ is a brane $L_a$ contained in the fiber above some vertex $a$ of $|\Delta^n|$. To define $\hom_{\cO_j}(L_a,L'_{a'})$ for two objects, we choose a Liouville connection\footnote{Note that we choose such a connection for every pair of objects; there is no global connection chosen.} along the edge of $|\Delta^n|$ from $a$ to $a'$, and we define the generators of $\hom$ to be given by parallel transport chords from $L_a$ to $L'_{a'}$. We artificially\footnote{This artifice will disappear upon passage to the wrapped setting; see~\cite{oh-tanaka-actions}.} impose an ordering by defining a partial order on the collection of objects, and we declare morphism complexes to be null when the target brane is not strictly larger than the domain brane. This directedness of $\cO_j$ is imposed to avoid dealing with a further layer of perturbations and choices (which one must deal with if one is to achieve transversality in the non-directed setting).

The $A_\infty$ operations are defined by counting holomorphic disks mapping into $j^* E$. We mention here that it is a priori not at all obvious how to articulate this counting problem---one must set up this count in a way compatible with different choices of $j$ and $j'$ for Theorem~\ref{theorem. O is Aoo} to hold. For this, we use a beautiful insight of~\cite{savelyev}, where Savelyev exhibits an operadically compatible map, for all $d \geq 1$, between the moduli of $(d+1)$-punctured holomorphic disks and the standard $d$-simplex. (See Section~\ref{section. curves operad}.)

To provide a more complete story of Floer theory in Liouville bundles, the present work also presents a careful consideration of continuation map methods in unwrapped Floer theory, both for individual Liouville manifolds  (or sectors) and their bundles. The main result along these lines (Theorem~\ref{thm:hcL=mu2ccL}) is a compatibility between two natural ways to define continuation maps.

Let us review the two ways to construct continuation maps.
Fix a Hamiltonian isotopy
    \eqnn
    \cL = \{L_s\}_{s \in [0,1]},
    \qquad
    L_0 = L,
    \qquad
    L_1 = L'
    \eqnd
of Lagrangians in $M$. To avoid clashing with later notation, we have set $L=L_0$ and $L'=L_1$.

Because we are working with Liouville manifolds and Liouville sectors, we require that $\cL$ be \emph{non-negative at infinity} (Definition~\ref{defn. non negative wrapping}). In the Liouville setting, these kinds of assumptions are made often to deal with the fact that our symplectic manifolds are non-compact---non-negativity allows us to apply a (strong) maximal principle and ensure compactness of the relevant moduli spaces.
(See \cite[Introduction]{oh-floer-continuity} for an early appearance of such a discussion.)

One standard way of defining a continuation map for Lagrangian Floer homology is by considering
the pseudoholomorphic curve equation for a strip with moving boundary conditions:
 \eqn\label{eq:moving-intro}
 	u: \RR \times [0,1] \to M,
	\qquad
        \begin{cases}
        {\frac{\partial u}{\partial\tau}} + J_{(\rho(1-\tau),t)} {\frac{\partial u}{\partial t}}=0 \\
        u(\tau ,0)\in K,\;\; u(\tau ,1)\in L_{\rho(1-\tau)}.
        \end{cases}
    \eqnd
Here, we have chosen another Lagrangian $K$ in general position with respect to $L_0$ and $L_1$. $\{J_{s,t}\}_{(s,t) \in [0,1]^2}$ is a family of almost complex structures, and $\rho:\RR \to [0,1]$ is a elongation function. (See Choice~\ref{choice. elongation rho}.) We refer the reader to Figure \ref{figure. continuation strip}; note that the moving boundary condition places $L$ near $\tau = \infty$, and places $L'$ near $\tau = -\infty$. (In particular, the isotopy evolves in the $-{\del/\del \tau}$ direction.),

\begin{figure}[ht]
	\eqnn
    			\xy
    			\xyimport(8,8)(0,0){\includegraphics[width=3in]{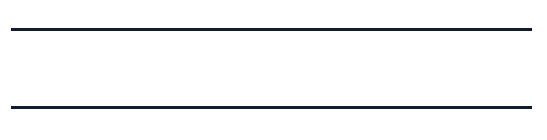}}
					,(-0.5,4.1)*{-\infty}
					,(8.4,4.1)*{+\infty}
					,(4,0.5)*{K}
					,(4,7.5)*{\cL}
					,(7.5,7.2)*{L}
					,(0.6,7.2)*{L'}
    			\endxy
    \eqnd
    \begin{figurelabel}\label{figure. continuation strip}
    A holomorphic strip with moving boundary condition $\cL$ at $t=1$ and fixed boundary condition $K$ at $t=0$. One defines a continuation map $CF(K,L) \to CF(K,L')$ by counting such strips.
    \end{figurelabel}
\end{figure}

The count of holomorphic strips $u$ results in a chain map
    \eqnn
    h^\rho_\cL: CF^*(K,L) \to CF^*(K,L')
    \eqnd
between the unwrapped Floer complexes. We warn that, because of the non-negativity constraint, the continuation map is not usually an isomorphism  in the unwrapped Fukaya category of a Liouville manifold  (or sector).

Another standard way to construct a continuation map (given the same isotopy $\cL$ as above) is to count holomorphic disks with one boundary puncture and moving boundary condition. Concretely, fix a point $z_0 \in \del D^2$ and choose another elongation function $\chi: \del D^s \setminus \{z_0\} \to [0,1]$. We let
	$$
\cM(D^2\setminus\{z_0\} ; \cL^\chi)
	$$
be the set of those maps
	\eqn
	v: D^2 \setminus\{z_0\}  \to M,
	\qquad
    \begin{cases}
    \delbar_J v = 0, \\
    \int_{D^2\setminus\{z_0\}} |dv|^2 < \infty, \\
    v(z) \in L_{\chi(z)} \quad \text{for } \, z \in \del D^2 \setminus \{z_0\}.
    \end{cases}
    \eqnd
The count of such disks defines an element
    \eqnn
    c_{\cL}^{\chi} \in CF^*(L,L').
    \eqnd
\begin{figure}[ht]
	\eqnn
    			\xy
    			\xyimport(8,8)(0,0){\includegraphics[width=2in]{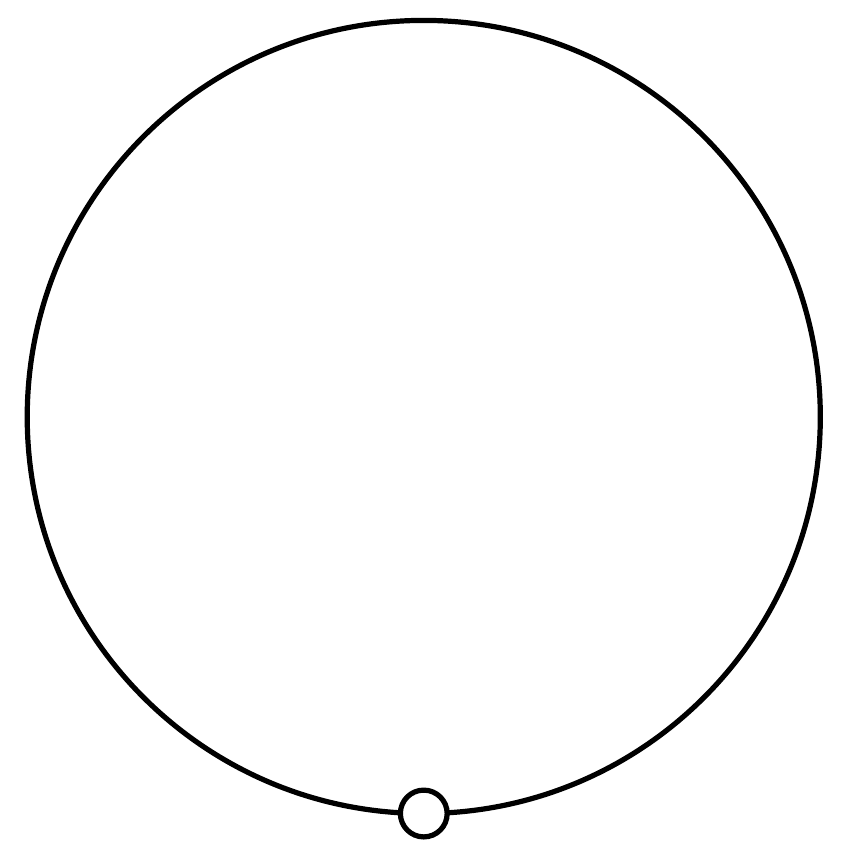}}
				,(4.1,1)*{-\infty}
				,(7.5,2)*{L}		
				,(0.4,2)*{L'}
				,(4,8.2)*{\cL}
    			\endxy
    \eqnd
    \begin{figurelabel}\label{figure. continuation disk}
    A holomorphic disk with one boundary puncture and with a moving boundary condition given by a non-negative isotopy $\cL$. The count of such disks gives rise to an element of $CF(L,L')$.
    \end{figurelabel}
\end{figure}

We prove (in the Liouville setting) that the two constructions above yield equivalent
elements in cohomology after applying the $\mu^2$ operation (i.e., after composing in the Fukaya category):

\begin{theorem}\label{thm:hcL=mu2ccL} Let $M$ be a Liouville manifold  (or sector). Let
$h_{\cL}^{\rho}$ and $c_{\cL}^{\chi}$ be as above. Then
        $$
        [h_{\cL}^{\rho}] = [\mu^2(c_{\cL}^{\chi}, -)]
        $$
        as maps $HF^*(K, L) \to HF^*(K,L')$.
\end{theorem}
This compatibility is well-known to experts, but a detailed proof is not easily found in the literature. We prove it here in the Liouville setting, and reassure the readers that the analogous results can be proven in other settings with appropriate modifications (as necessary) taken to prove $C^0$ and energy estimates.

Theorem~\ref{thm:hcL=mu2ccL} implies its analogue in the setting of Liouville bundles, where one allows $K$ to be in a possibly different fiber. However, the proof of this implication requires a detailed proof of the
$C^0$-estimates and the energy
estimates needed for the compactness study of moduli spaces of pseudoholomorphic sections; in proving a bundle analogue of Theorem~\ref{thm:hcL=mu2ccL}, we also highlight
the roles of nonnegativity of the isotopy and of the ``pinchedness'' (from below) of the curvature of
Liouville bundles. (See Theorem~\ref{thm:C0estimate+CL} for the precise statement.)

\subsection{More motivations}\label{section. motivating continuation compatibility}

While Theorem~\ref{thm:hcL=mu2ccL} is of obvious practical importance, let us mention how we utilize the result in later work.

Continuation maps play a central role in passing from the unwrapped, directed Fukaya category $\cO_j$ to the {\em wrapped} Fukaya category. In~\cite{oh-tanaka-actions} (following unpublished ideas of Abouzaid-Seidel as described in~\cite{gps}), we define the wrapped Fukaya category $\cW_j$ by localizing $\cO_j$ along continuation maps. Informally, this means that we algebraically declare Lagrangians related by a non-negative Hamiltonian isotopy to be equivalent objects. As it turns out, this has the consequence that morphisms in the wrapped category can be computed geometrically, by constructing a complex that can be modeled as an increasing union over Floer complexes of larger and larger non-negative Hamiltonian isotopies. We refer the reader to~\cite{gps} and~\cite{oh-tanaka-actions} for details.

Now, for our purposes: We must utilize our compatibility result in Theorem~\ref{thm:hcL=mu2ccL} to actually give a workable description of $\cW_j$ and to perform computations in $\cW_j$. This allows us to promote Theorem~\ref{theorem. O is Aoo} to a functor $\cW$ from $\simp(B)$ sending all $\iota$ to an {\em equivalence} of $A_\infty$-categories (not just fully faithful functors); this exhibits the local system, and the proof of Teleman's conjecture. These details are in~\cite{oh-tanaka-actions}.

Let us now remark why we view Theorem~\ref{thm:hcL=mu2ccL} as giving a ``workable'' framework for setting up $\cW_j$. The issue is that the strip definition of continuation maps does not naively define a map form $L$ to $L'$ in the directed, unwrapped category. Usually, one produces an element of $HF(L,L')$ out of a Hamiltonian isotopy by proving the naturality of $h_{\cL}$ in the $K$ variable, then invoking the Yoneda embedding. But to geometrically interpret the Yoneda embedding, one needs a geometric interpretation of the unit map of an object---such an interpretation is {\em unavailable} in the directed Fukaya category, as the identity morphism is constructed by formal algebra, rather than by defining the actual endomorphism Floer complex of a brane.

Finally, we hope that our description of what {\em choices} are needed to specify a collection of unwrapped Fukaya categories living over a Liouville bundle, while exhibiting the existence of compatible choices (to simultaneously output the usual $A_\infty$-relations and to jive with face maps of the simplices $j$) may be of interest to readers looking to exploit Floer theory in the bundle setting.

{\bf Acknowledgments.}
The first author is supported by the IBS project IBS-R003-D1.
The second author was supported by
IBS-CGP in Pohang, Korea and
the Isaac Newton Institute in Cambridge, England,
during the preparation of this work. This material is also based upon work supported by the National Science Foundation under Grant No. DMS-1440140 while the second author was in residence at the Mathematical Sciences Research Institute in Berkeley, California, during the Spring 2019 semester.

\section{Liouville bundles and connections}

\subsection{Liouville domains}

The notion of Liouville domain will not make a frequent appearance in our work; but it is a convenient stepping stone to the notion of Liouville manifold.

\begin{defn}\label{defn:liouville vector field Z}
Fix a symplectic manifold $(M,\omega)$.
A vector field $Z$ on $M$ is said to be a {\em Liouville vector field} if the Lie derivative of $\omega$ along $Z$ is $\omega$ itself:
	\eqn\label{eqn:Liouville Z}
\cL_Z\omega = \omega.
	\eqnd
Given a Liouville vector field $Z$, its flow will be called the {\em Liouville flow}.
\end{defn}

\begin{defn}\label{defn. liouville form}
Given a Liouville vector field $Z$, let $\lambda$ be the 1-form defined by the equation
	\eqn\label{eqn:Liouville form lambda}
	\lambda = \omega(Z,\cdot).
	\eqnd
We call $\lambda$ the {\em Liouville form}.
\end{defn}

\begin{remark}\label{remark:liouville-implies-exact}
Fix a vector field $Z$ and its dual $\lambda$ as in Equation~\eqref{eqn:Liouville form lambda}. Then
Equation~\eqref{eqn:Liouville Z} is equivalent to the condition that $\lambda$ is an anti-derivative of $\omega$:
	\eqnn
	\omega = d \lambda.
	\eqnd
In particular, any symplectic manifold equipped with a Liouville vector field is an exact symplectic manifold. Conversely, given a 1-form $\lambda$ satisfying $d\lambda = \omega$, one sees that the dual vector field defined by ~\eqref{eqn:Liouville form lambda} is automatically a Liouville vector field.
\end{remark}

\begin{defn}\label{defn:liouville-domain}
A \emph{Liouville domain} is a compact symplectic manifold $W$ with boundary, equipped with a Liouville vector field $Z$ which points strictly outward along the boundary.
\end{defn}

\begin{remark}
By the exactness witnessed in Remark~\ref{remark:liouville-implies-exact}, any Liouville domain $W$ must have non-empty boundary unless $W$ is 0-dimensional.
\end{remark}

\begin{notation}[The boundary $\del_\infty W$ of a Liouville domain]
Fix a Liouville domain $W$ (Definition~\ref{defn:liouville-domain}). We let $\del_\infty W$ be the boundary manifold.
\end{notation}

\begin{remark}\label{remark. boundary of liouville domain is contact}
Let $W$ be a Liouville domain. It follows that $\xi = \ker \lambda|_{\del_\infty W}$ is a contact structure on $\del_\infty W$.
\end{remark}

\begin{remark}[Co-orientation]
Recall that a {\em co-orientation} on a contact manifold is a choice of 1-form whose kernel is equal to the contact distribution.  We see that the boundary $\del_\infty W$ of any Liouville domain is a contact manifold co-oriented by $\lambda|_{\del_\infty W}$.
\end{remark}

\subsection{Symplectizations}

\begin{notation}[Symplectization $SY$ of a contact manifold]\label{notation. symplectization}
Given a co-oriented contact manifold $(Y,\alpha)$, its symplectization $SY$ is the manifold
	\eqnn
	SY = \RR \times Y
	=
	\{(s,y)\}.
	\eqnd

We equip $SY$ with the Liouville form (Definition~\ref{defn. liouville form})
	\eqnn
	e^s \pi^* \alpha
	\eqnd
where $\pi: SY \to Y$ is the projection map.
\end{notation}

\begin{notation}[$r$ and $s$]\label{notation. r in symplectization}
We will often use the change of coordinates
	\eqnn
	r = e^s.
	\eqnd
\end{notation}

\subsection{Liouville manifolds}

We now pass from the setting of a Liouville domain to a ``manifold-with-conical-end'' setting:

\begin{notation}\label{notation. Liou equivalence}
Let $M$ be a smooth manifold. We define an equivalence relation on the set of smooth 1-forms on $M$ as follows: We say $\theta \sim \theta'$ if and only if there exists a smooth, compactly supported function $f: M \to \RR$ for which
	\eqnn
	\theta = \theta' + df.
	\eqnd
We let
	$
	[\theta]_{\Liou}
	$
denote the equivalence class of $\theta$.
\end{notation}

\begin{defn}[Liouville manifold]\label{defn:liouville manifold}
Fix the data of a pair $(M,[\theta]_{\Liou})$, where $[\theta]_{\Liou}$ is as in Notation~\ref{notation. Liou equivalence}. We say this pair is a \emph{Liouville manifold} if for some (and hence any) choice $\theta \in [\theta]$, the pair $(M,\theta)$ is a completion of
a Liouville domain.
\end{defn}

\begin{remark}\label{remark. completion}
Let us explain what we mean by a completion. We mean there exists a compact, co-oriented contact manifold $(Y,\alpha)$, and a map from the `positive half' of $SY$
	\eqn\label{eq:iota}
\iota: \RR_{s \geq 0} \times Y \to M
	\eqnd
such that
	\enum
	\item $\iota$ respects Liouville forms, i.e., $\iota^*(\theta) = e^s \pi^*\alpha$,
	\item $\iota$ is a diffeomorphism (of manifolds with boundary) onto its image, and
	\item The complement $M \setminus \iota(\RR_{>0} \times Y)$ is a Liouville domain when equipped with (the restriction of) $\theta$.
	\enumd
\end{remark}

\begin{remark}
We have reserved $M$ to denote possibly-non-compact exact symplectic manifolds, while $W$ always denotes (compact) Liouville domains.
\end{remark}

\begin{remark}
It is common to define a Liouville manifold as equipped with a choice of $\theta$, rather than just of $[\theta]_{\Liou}$. While we utilize particular choices of $\theta$ to perform certain geometric constructions, a particular choice conceals the appropriate notion of automorphism. (See Definition~\ref{defn:liouville-automorphism}.)
\end{remark}

\begin{remark}
One may pass freely between a Liouville domain to a Liouville manifold (by completion), and vice versa (by choosing an $\iota$ as in~\eqref{eq:iota}). However, the notion of Liouville manifold will be more canonical---e.g., less choice-dependent---in our applications.
\end{remark}

\subsection{Liouville sectors}\label{section. liouville sectors}

\begin{remark}
The notion of Liouville sector is due to~\cite{gps}, and extends the notion of Liouville manifold to the setting with boundary.
\end{remark}

Just as Liouville manifolds are naturally presented as completions of exact symplectic manifolds with boundary, a Liouville sector is naturally the completion of an exact symplectic manifold $W$ with {\em corners}.

\begin{defn}[Liouville domain with convex boundary]\label{defn. convex domain}
Fix a compact exact symplectic manifold $(W,\theta)$ with corners. We let $D W$ denote the entire boundary of $W$---i.e., the union of all faces and corners of $M$.

We say the pair $(W,\theta)$ is a {\em Liouville domain with convex boundary} if the following are satisfied:

\newenvironment{convex-domain-props}{
	  \renewcommand*{\theenumi}{(CB\arabic{enumi})}
	  \renewcommand*{\labelenumi}{(CB\arabic{enumi})}
	  \enumerate
	}{
	  \endenumerate
}

\begin{convex-domain-props}
\item \label{item. DW has two boundaries} (There are two kinds of boundary.) $D W$ admits two smooth, codimension zero submanifolds-with-boundary $\del W$ and $\del_\infty W$ such that
	\eqnn
	\del_\infty W \cap \del W
	\eqnd
	is precisely the locus of corners of $W$, and
	\eqnn
	DW = \del_\infty W \bigcup_{\del_\infty W \cap \del W} \del W.
	\eqnd
\item ($\del_\infty W$ is contact.) We demand that the Liouville vector field $Z$ is strictly outward-pointing with respect to $\del_\infty W$. In particular, $\theta|_{\del_\infty W}$ renders $\del_\infty W$ a (co-oriented) contact manifold with boundary.
\item ($\del W$ is convex.) \label{item. del W convex} We demand that there exists a smooth function $I: W \to \RR$ satisfying $ZI = \alpha I$ for some $\alpha> 0$ whose Hamiltonian flow along $\del W$ is strictly outward pointing. (See \cite[Definition 2.4]{gps}.)
\item\label{item. Z parallel to del W} (The $\del W$ boundary can be extended along $Z$.) For simplicity, we will further assume that in some neighborhood of $\del_\infty W$, $Z$ is contained in $T(\del W)$. (So near $\del_\infty W$, $Z$ is tangent to $\del W$. One can always deform $\theta$ so that this is the case).
\end{convex-domain-props}
\end{defn}

\begin{remark}
Let $(W,\theta)$ be a Liouville domain with convex boundary (Definition~\ref{defn. convex domain}).
Using the notation from ~\ref{item. DW has two boundaries}, one may informally think of $\del W$ as the {\em wall} of $W$, while one may think of $\del_\infty W$ as the {\em ceiling}. (There are no floors.)
	
Based on~\ref{item. Z parallel to del W}, the reader may imagine that the Liouville flow pushes the ceiling higher toward the sky, in a way such that the walls may similarly be extended upward.	The Liouville flow may push on the walls inwards or outwards, but it only does so away from a neighborhood of the ceiling.
\end{remark}

The reader should compare the following to the definition of Liouville domain (Definition~\ref{defn:liouville manifold}). It is equivalent to Definition~2.4 of~\cite{gps}.

\begin{defn}
Fix a pair $(M,[\theta]_{\liou})$ where $M$ is a smooth manifold with boundary. We say that $(M,[\theta]_{\liou})$ is a {\em Liouville sector} if, for some (and hence all) $\theta \in [\theta]_{\liou}$, the pair $(M,\theta)$ is the completion of a Liouville domain with convex boundary.
\end{defn}

\begin{remark}
By a completion, we mean the data of a co-oriented contact manifold $(Y,\alpha)$ with boundary, and a map $\iota: \RR_{s \geq 0} \times Y \to M$ such that the appropriate analogues of Remark~\ref{remark. completion} are satisfied. In particular, $\iota$ is a diffeomorphism of smooth manifolds with corners, and the restriction of $\theta$ exhibits the complement $M \setminus \iota(\RR_{>0} \times Y)$ as a Liouville domain with convex boundary.
\end{remark}

\begin{remark}
Let $(M,[\theta]_{\liou})$ be a Liouville sector. Then $M$ only has one ``type'' of boundary, $\del M$, which one may think of as an extension of the wall $\del W$ by the Liouville flow.
\end{remark}

\begin{remark}
Henceforth, we use the term Liouville sector with the understanding that if a Liouville sector has empty boundary, then it is in particular a Liouville manifold.
\end{remark}

\subsection{Eventually conical branes}

\begin{defn}
A subset $A \subset M$ is called {\em conical near infinity} if for some (and hence all) $\theta \in [\theta]_{\liou}$, and for some compact subset $K$, the complement $A \setminus K$ is closed under the positive Liouville flow.
\end{defn}

There are standard decorations one should put on Liouville manifolds (or sectors) and their Lagrangians to obtain a $\ZZ$-graded, $\ZZ$-linear Fukaya category---for example, gradings and Pin structures. We assume these structures to be chosen throughout. To that end:

\begin{defn}\label{defn. branes}
Let $M$ be a Liouville manifold (or sector).
A {\em brane} is a conical-near-infinity Lagrangian $L \subset M$ equipped with the relevant brane decorations.
\end{defn}

Because brane structures will not feature prominently in this work, we refer the reader to~\cite{seidel-book} for the basics, and to Section~2.3 of~\cite{oh-tanaka-actions} for how the structure group of a Liouville bundle changes as one demands different brane structure.

\subsection{Non-negative isotopies}
We recall the notion of a nonnegative exact Lagrangian isotopy.

\begin{defn}[Non-negative isotopy]\label{defn. non negative wrapping}
Fix an exact Lagrangian isotopy $j: L \times [0,1]_t \to M$ through conical-near-infinity Lagrangians. (In particular, this induces an isotopy of Legendrians inside $\del_\infty M$.) We say this is a {\em non-negative wrapping}\footnote{In~\cite{gps}, this notion is called a positive wrapping (see Definition 3.20 of loc. cit.).}, or a {\em non-negative isotopy} if for some (and hence any) choice of Liouville form $\theta$ on $M$, we have the following outside a compact subset of $L$:
	\eqnn
	\theta (Dj(\del_t)) \geq 0.
	\eqnd
Put another way, the flow of $L$ in $\del_\infty M$ is non-negative with respect to the contact form induced by $\theta$.
\end{defn}

\subsection{Liouville automorphisms}

\begin{defn}[Liouville automorphisms]
\label{defn:liouville-isomorphism-M}
\label{defn:liouville-automorphism}
Let $M_i $,
$i = 0,\, 1$, be Liouville sectors. A {\em Liouville isomorphism} from $M_0$ to $M_1$ is a diffeomorphism $\phi: M_0 \to M_1$ satisfying
	\eqnn
	\phi^*[\theta_1]_{\Liou} = [\theta_0]_{\Liou}.
	\eqnd
(See Notation~\ref{notation. Liou equivalence}.)
If $M_0 = M_1$, we call $\phi$ a Liouville {\em automorphism}.
\end{defn}

\begin{definition}\label{defn:Liouaut}
Let $M$ be a Liouville sector.
We let
	\eqnn
	\Liouauto(M)
	\eqnd
denote the topological group of Liouville automorphisms of $M$. It is topologized as a subspace of $C^\infty(M,M)$ with the strong Whitney topology.
\end{definition}

\begin{warning}
The choice of $[\theta]_{\liou}$ is not explicit in the notation $\Liouauto(M)$.
\end{warning}

\subsection{Liouville bundles}

\begin{defn}[Liouville bundle]\label{defn. liouville bundle}
Fix a Liouville manifold  (or sector) $M$. A {\em Liouville bundle} with fiber $M$ is the choice of a smooth $M$-bundle $p: E \to B$, together with a smooth reduction of the structure group from $\diff(M)$ to $\Liouauto(M)$.
\end{defn}

\begin{remark}
Definition~\ref{defn. liouville bundle} applies when $p: E \to B$ is a smooth map of diffeological spaces (see \cite[Section 3.2]{oh-tanaka-smooth-approximation}), or smooth manifolds with corners. By a {\em smooth} reduction of structure group, we mean that for an open cover, the specified transition maps $U_{\alpha\beta} \to \Liouauto(M)$ must be smooth (in the sense of the diffeology on $\Liouauto(M)$ and the diffeology of $B$).
\end{remark}

\begin{remark}\label{remark. liouville bundles different structure group}
If one is interested in a Liouville bundle with a structure group allowing one to trivialize brane structures over simplices, one should demand a smooth reduction of structure group not to $\Liouauto$, but to another structure group $\Liouaut$. We refer the reader to Section~2.3 of~\cite{oh-tanaka-actions} for possible other structure groups. We also note that the smoothness of reduction may now be tested by composing a map to $\Liouaut$ with the natural projection $\Liouaut \to \Liouauto$.
\end{remark}

\begin{notation}[$\del E$]\label{notation. del E}
Let $E \to B$ be a Liouville bundle whose fibers are Liouville manifolds  (or sectors), and suppose these fibers are all isomorphic to some Liouville manifold (or sector) $M$. We denote by
	\eqnn
	\del E \to B
	\eqnd
the induced fiber bundle whose fibers are diffeomorphic to $\del M$. Note that we use the symbol $\del E$ regardless of whether the base $B$ has boundary, corners, et cetera.
\end{notation}

\begin{remark}[$\Theta$]\label{remark. can choose liouville form}
Let $E \to B$ be a Liouville bundle and suppose $E$ and $B$ are both smooth manifolds, possibly with corners. First let $B$ be smoothly contractible. Then there exists a global choice of 1-form
	\eqnn
	\Theta \in \Omega^1(E;\RR)
	\eqnd
such that:

\newenvironment{Theta-prop}{
	  \renewcommand*{\theenumi}{($\Theta$\arabic{enumi})}
	  \renewcommand*{\labelenumi}{($\Theta$\arabic{enumi})}
	  \enumerate
	}{
	  \endenumerate
}

\begin{Theta-prop}
	\item\label{item. Theta prop} for every $b \in B$, the fiberwise restriction $\Theta|_{E_b}$ defines a 1-form on the fiber $E_b$ exhibiting $E_b$ as a Liouville completion.
\end{Theta-prop}
By the paracompactness of $B$ and a partition of unity argument, we thus have a global 1-form $\Theta$ on $E \to B$  satisfying property~\ref{item. Theta prop} for arbitrary base manifolds $B$. We call $\Theta$ a \emph{fiberwise Liouville form} for $E \to B$, or just a Liouville form as long as there is no danger of confusion.
\end{remark}

\begin{remark}\label{remark. Theta space contractible}
Fix a Liouville bundle $E \to B$ where $B$ (and hence $E$) is a smooth manifold, possibly with corners. Then the space of $\Theta$ satisfying~ \ref{item. Theta prop} is convex, and in particular, smoothly contractible.
\end{remark}

\begin{example}\label{example. universal bundle}
Fix a Liouville manifold  (or sector) $M$. There will be two classes of Liouville bundles of interest associated to $M$.

The first is the {\em universal} Liouville bundle. (See \cite[Section 2.7]{oh-tanaka-smooth-approximation}.) This is constructed from the universal principle bundle
	\eqnn
	\widehat{E\Liouauto(M)}
	\to
	\widehat{B\Liouauto(M)}
	\eqnd
by taking the induced principle $M$-bundle 	
	\eqnn
	E \to \widehat{B\Liouauto(M)}
	\eqnd
(whose structure group is canonically smoothly reduced to $\Liouauto(M)$). The reader may appreciate that in, Remark~\ref{remark. can choose liouville form} we assumed $B$ is a smooth manifold---in the example of $B = \widehat{B\Liouauto(M)}$, $B$ is not a manifold.

The second main example is given by taking a smooth map from an extended smooth simplex
	\eqnn
	j: |\Delta^n_e| \to B
	\eqnd
where $|\Delta^n_e|$ is the affine hyperplane defined by the equation $\sum_{i=0}^n t_i = 1$.
(We refer to \cite[Definition 2.5]{oh-tanaka-smooth-approximation} for a discussion on why we use the extended smooth simplex.) We then pull back the bundle $E \to B$ to obtain a smooth Liouville bundle $j^*E \to |\Delta^n_e|$; in particular, one obtains another smooth Liouville bundle by restricting further to the standard $n$-simplex $|\Delta^n| \subset |\Delta^n_e|$.
\end{example}

\subsection{Connections on bundles}\label{section. Liouville connection on bundles}

\begin{defn}[Connection]\label{defn. connection}
Let $\pi: E \to B$ be a smooth fiber bundle.
Recall that a (Ehresmann) connection is a choice of splitting
	\eqn\label{eqn. ehresmann connection}
	TE \cong HTE \oplus VTE
	\eqnd
where $VTE = \ker(d\pi)$. As usual we will call $HTE$ the horizontal distribution (associated to the connection).
\end{defn}

Fix a Liouville bundle $\pi: E \to B$ over a smooth manifold $B$, and equip $E$ with a choice of global 1-form $\Theta \in \Omega^1(E)$ as in Remark~\ref{remark. can choose liouville form}. Then one has a natural connection on $\pi: E \to B$, defined as follows:

\begin{defn}\label{defn. theta connection}
The {\em connection associated to $\Theta$} is the subbundle of $TE$ consisting of those tangent vectors $x$ for which
	\eqnn
	VTE \subset \ker \left( d\Theta(-, x)\right).
	\eqnd
That is, any vertical tangent vector is annihilated when paired with $x$ using $d\Theta$.
\end{defn}

In particular, any Liouville bundle equipped with  $\Theta$ as in Remark~\ref{remark. can choose liouville form} has a well-defined notion of parallel transport along smooth curves.

\subsection{Almost complex structures}

\begin{defn}\label{defn. J conical near infinity}
Let $E \to B$ be a Liouville bundle. Let $\JJ$ be a smooth choice of fiber-wise almost complex structures on $E$.

We say that $\JJ$ is {\em conical near infinity} if for some (hence any ) choice of $\Theta$ (as in Remark~\ref{remark. can choose liouville form}), there exists a subset $K \subset E$, proper over $B$, such that the following holds:
\enum
	\item For each $b \in B$, $K \cap E_b$ is a Liouville domain (exhibiting the fiber $E_b$ as the Liouville completion of $K \cap E_b$), and
	\item Writing $E_b$ as the completion of $K \cap E_b$ with conical coordinate $r=e^s$, we have that	
	\eqnn
	\Theta|_{E_b} \circ \JJ|_{E_b} = d(e^s)
	\eqnd
	for $s>>0$.
\enumd
\end{defn}

\begin{remark}
If $\JJ$ is conical near infinity (Definition~\ref{defn. J conical near infinity}), it follows that along each fiber of $E \to B$, the Lie derivative of $\JJ|_{E_b}$ with respect to the Liouville flow vanishes outside some compact subset (for example, outside of $K \cap E_b$).
\end{remark}

\begin{example}\label{example. J for single Liouville manifold}
If $B$ is a point, then a choice of $\JJ$ as in Definition~\ref{defn. J conical near infinity} is a choice of conical-near-infinity almost-complex structure $J$ on the fiber Liouville manifold (or sector), in the usual sense.
\end{example}

\begin{notation}[$\cJ$]\label{notation. cJ vector bundle}
For every $b \in B$, let $\cS_b \subset B$ denote the Riemann surface containing $b$. Then over $E$ there is a natural bundle
	\eqnn
	\cJ \to E
	\eqnd
whose fibers above $x \in E$ consist of almost-complex structures on the vector bundle
	\eqn\label{eqn. disk-wise complex structures}
	(d\pi)^{-1}(T_b \cS_b) \subset T_x E.
	\eqnd
\end{notation}
	
\begin{remark}
Here is another description of $\cJ$. Let $B = \overline{\cS}_{d+1}^\circ \to \overline{\cR}_{d+1}$ denote the projection map for the universal family of curves, and let $\cH \subset TB$ denote the vertical tangent bundle of this projection. Fix further a Liouville form $\Theta$ on $\pi: E \to B$, so that we have an induced splitting $TE \cong HTE \oplus VTE$ as in~\eqref{eqn. ehresmann connection}. By the identification $HTE \cong \pi^* TB$, we have an induced subbundle $\pi^* \cH \oplus VTE \subset TE$. $\cJ$ is the bundle whose global sections are choices of almost-complex structures on $\pi^* \cH \oplus VTE$.
\end{remark}

\begin{defn}[$\JJ$ Suitable for counting sections]\label{defn. suitable for sections}
Let $B = \overline{\cS}_{d+1}^\circ$ and fix a Liouville bundle $\pi: E \to B$. Let $\cJ$ be the bundle from Notation~\ref{notation. cJ vector bundle}. We say that a global section $\JJ$ of $\cJ$ is {\em suitable for counting sections} when the following are satisfied:
\enum
\item \label{item. suitable. holomorphic projection} For every member of the universal family $\cS_r \subset \overline{\cS}_{d+1}^\circ$, let $E_r \to \cS_r$ denote the pulled back Liouville bundle. We demand that the projection map is holomorphic---that is,
	\eqnn
	d\pi \circ \JJ|_{E_r} = j_r \circ d\pi.
	\eqnd
(Here, $j_r$ is the complex structure on $\cS_r$.)
\item  $\JJ$ preserves the vertical tangent space $VTE$, and $\JJ|_{VTE}$ is a conical-near-infinity almost-complex structure for the bundle $E \to B$ as in Definition~\ref{defn. J conical near infinity}.
\item\label{item. suitable. eventually preserves horizontal} Finally, we demand that for some (and hence any) choice of global Liouville form $\Theta$ on $E$ as in Remark~\ref{remark. can choose liouville form}, there exists a subset $K \subset E$ (independent of  $r \in \overline{\cR}_{d+1}$), proper over $B$, such that $\JJ(HTE_r) = HTE_r$. (Here, $HTE_r$ is the horizontal tangent space induced by pulling back the connection on $E$ to a connection on $E_r$.)
\enumd
By abuse of notation, we will refer to $\JJ$ also as a choice of almost-complex structure. (Even though, strictly speaking, $\JJ$ only defines almost-complex structures on each $E_r$, and not on all of $E$.)
\end{defn}

\begin{remark}
Let $\JJ$ be an almost-complex structure suitable for counting sections (Definition~\ref{defn. suitable for sections}). Choose a splitting $T_xE_r \cong VT_xE \oplus T_{\pi(x)}\cS_r$, condition~\ref{item. suitable. holomorphic projection}. says that $\JJ_x$ may be written as a block triangular matrix. Condition~\ref{item. suitable. eventually preserves horizontal}. says that, outside controlled, fiber-wise compact subset, $\JJ_x$ is block diagonal. In particular the space of $\JJ$ is seen to be contractible.
\end{remark}

\subsection{Defining functions and barriers on families}

Let us recall from~\cite{gps} that if $M$ is a Liouville sector, there exists a smooth map
	\eqnn
	\pi: \nbhd(\del M) \to \CC_{\Re \geq 0}
	\eqnd
from a neighborhood of $\del M$ to the complex numbers with positive real coordinate. This (possibly non-surjective) map satisfies the following:
	\enum
	\item The imaginary coordinate of $\pi$ defines a smooth, linear-near-infinity function $I$ whose Hamiltonian vector field is outward pointing at $\del M$. (This is called a ``defining function'' in~\cite{gps}.)
	\item Moreover, there is a contractible space of almost-complex structures $J$ on $M$, compatible with the Liouville structure of $M$, such that $\pi$ is $J$-holomorphic.
	\enumd
	
\begin{remark}\label{remark. barrier}
This allows one to use a standard ``barrier'' type argument using the open mapping theorem to conclude the following: If a holomorphic curve $u: S \to M$ has boundary Lagrangians supported away from $\del M$, then the image of $u$ must be bounded away from $\del M$ as well. (See 2.10.1 of~\cite{gps}.) This is the main utility of the definition of Liouville sector, and in particular of the defining function $I$. (Informally, while $I$ defines the imaginary coordinate of $\pi$, its negative Hamiltonian flow-time away from $\del M$ defines the real coordinate---see the proof of Proposition~2.24 of~\cite{gps}.)
\end{remark}

\begin{remark}\label{remark. barrier in families}
Now if $E \to B$ is a Liouville bundle of Liouville sectors over a smooth manifold, a partition of unity argument defines a global function $\pi: \nbhd(\del E) \to \CC_{\Re \geq 0}$ whose imaginary part restricts on each fiber to a defining function $I$. ($\del E$ is defined in Notation~\ref{notation. del E}.) By contractibility of the space of almost-complex structures, one can choose $\JJ$ on $E$ such that
	\eqnn
	D\pi|_{VTE} \circ \JJ = j_{\CC_{\Re \geq 0}} \circ D\pi|_{VTE}
	\qquad
	\text{on $\nbhd(\del E)$}.
	\eqnd
(Here $VTE = \ker D \pi$ is the vertical tangent bundle of $E$.) In particular, given any map $u: S \to E$ which is holomorphic with respect to $J$, the composite $\pi \circ u: S \to \CC_{\Re \geq 0}$ is holomorphic, and the same barrier argument as in Remark~\ref{remark. barrier} shows that the image of $u$ must be bounded away from the boundary of each fiber Liouville sector.

In particular, if one has a priori $C^0$ bounds on the strip-like ends of $S$, then one has an a priori $C^0$ bound on $u$ given the boundary conditions.
\end{remark}

\clearpage

\section{Simplices, and families of disks}

\begin{notation}[Simplices]\label{notation. standard simplices}\label{notation. extended simplices}
Fix an integer $d \geq 0$. We let $|\Delta^d|$ denote the standard topological $d$-dimensional simplex, given by the subset of those $(t_0,\ldots,t_d) \in \RR^{d+1}$ satisfying $t_i \geq 0$ and $\sum t_i = 1$. More generally, given any linearly order set $A$, we let $|\Delta^A|$ denote the subset of $\RR^A$ given by those $(t_a)_{a \in A}$ satisfying $t_a \geq 0$ and $\sum_{a \in A} t_a =1$.

We will sometimes refer to $|\Delta^A|$ as the geometric realization of $A$.

The {\em extended} $d$-simplex is the space $|\Delta^d_e| \subset \RR^{d+1}$ of those $(t_0,\ldots,t_d) \in \RR^{d+1}$ satisfying $\sum t_i = 1$. It is abstractly homeomorphic to $\RR^d$.
\end{notation}

\begin{defn}[$i$th vertex]\label{defn. ith vertex}
Let $|\Delta^d|$ be a standard simplex. Given $i \in \{0, \ldots, d\}$, the {\em $i$th vertex of $|\Delta^d|$} is the unique point whose $i$th coordinate is equal to 1.

Likewise, if $|\Delta^d_e|$ is the extended simplex, the $i$th vertex is the same point (with $i$th coordinate 1 and other coordinate 0).
\end{defn}

\begin{remark}[Standard and extended simplices]\label{remark. standard vs extended simplices}
Because the natural inclusion $|\Delta^n| \to |\Delta_e^n|$ is a smooth map (from a manifold with corners), it will make sense to pullback smooth objects living over an extended simplex to a standard simplex.

Finally, we will do our best to use the letter $h$ to denote maps from a standard simplex, and the letter $j$ to denote maps from the extended simplex:
	\eqnn
	h: |\Delta^n| \to B,
	\qquad
	j: |\Delta_e^n| \to B.
	\eqnd
\end{remark}

We review Savelyev's observation from~\cite{savelyev} that two fundamental objects of our fields---(i) universal families of holomorphic disks with $k + 1$ boundary punctures, and (ii) standard simplices---have compatible operadic structures.

Let us explain how we use this observation. Our goal (Theorem~\ref{theorem. O is Aoo}) is to associate, for every smooth map
	\eqnn
	j: |\Delta_e^n| \to B
	\eqnd
and every Liouville bundle $E \to B$, a non-wrapped Fukaya category $\cO_j$. This means that given $j$ and a Liouville bundle, we must associate a $d$-ary $A_\infty$ operation for every $(d+1)$-tuple of objects. In a way we make explicit later, this is done by taking a map
	\eqnn
	|\Delta^d| \to |\Delta^n|
	\eqnd
induced by the $(d+1)$-tuple of objects, and noticing that the $d$-simplex $|\Delta^d|$ itself can be (modulo a neighborhood of the boundary) {\em identified with the total space} of the universal family
	\eqnn
	\cS_{d+1} \to \cR_{d+1}
	\eqnd
of $(d+1)$-ary holomorphic disks (see Notation~\ref{notation:universal-family} and Remark~\ref{remark:simplices are homeo to families of disks}). At the very end of the present section, we will choose such an identification $|\Delta^d| \approx \cS_{d+1}$ once and for all. (Here, we use $\approx$ rather than $\cong$ to indicate that this is an identification modulo boundary.)

Roughly speaking, we will then define the $d$-ary operation $m_d$ to be given by counts of holomorphic sections $u$ (with Lagrangian boundary conditions) from fibers of $\cS_{d+1}$ to the bundle obtained by pulling back $E$ along the composite $\cS_{d+1} \approx |\Delta^d| \to |\Delta^n| \to |\Delta_e^n| \to B$:
	\eqnn
	\xymatrix{
	E|_{\cS_r} \ar[rrr] \ar[d]
	&&& E \ar[d] \\
	\cS_r \ar[r] \ar@{-->}@/^/[u]^u
		& \cS_{d+1} \approx |\Delta^d| \ar[r]
		& |\Delta^n| \to |\Delta_e^n| \ar[r]
		& B.
	}
	\eqnd
Here, $\cS_r \subset \cS_{d+1}$ is a holomorphic disk with $d+1$ boundary marked points; it is the fiber above $r \in \cR_{d+1}$.

The reader may now appreciate that for such counts to satisfy the $A_\infty$-relations, one must impose some compatibilities on the structures chosen to define these operations---and especially the identifications $\cS_{d+1} \approx |\Delta^d|$---as one approaches the boundary moduli of nodal disks. To articulate these compatibilities, we will also be forced to choose maps
	\eqnn
	\nu_{\beta}: \overline{\cS}_{d+1}^\circ \to |\Delta^n|
	\eqnd
for each simplicial map $\beta: |\Delta^d| \to |\Delta^n|$. (See Notation~\ref{notation:universal-family} for the notation $\overline{\cS}^\circ_{d+1}$.) Moreover, we would later like these non-wrapped Fukaya categories to be functorial in the choice of $j$, meaning that if we have a simplicial inclusion $|\Delta_e^{n'}| \subset |\Delta_e^n|$,  the composite map $j': |\Delta_e^{n'}| \to |\Delta_e^n| \xrightarrow{j} B$, induces a functor $\cO_{j'} \to \cO_{j}$ of non-wrapped Fukaya categories.
This imposes further compatibilities on our choices.

The main purpose of this section is to define what these compatibilities are in terms of the maps $\nu_\beta$, which for the special case of $\beta = \id$ recovers the identifications $\cS_{d+1} \cong |\Delta^d|$. This is given in Definition~\ref{defn. natural system}. We record the existence of such choices in Proposition~\ref{prop:natural systems exist}.

\subsection{Universal families of curves and gluing along strip-like ends}\label{section. curves operad}

\begin{notation}[$\overline{\cR}, \overline{\cS}, \overline{\cS}^\circ$.]\label{notation:universal-family}
Let $\overline{\cR}_{d+1}$ denote the compactified moduli space of holomorphic disks with $d+1$ boundary punctures; we demand that one of these boundary punctures is distinguished, and we refer to it as the outgoing marked point, or the 0th marked point. Using the boundary orientation of a holomorphic disk, any other marked point may uniquely be labeled as the $i$th marked point for some $1 \leq i \leq d$.

We let $\overline{\cS}_{d+1} \to \overline{\cR}_{d+1}$ denote the universal family of (possibly nodal) disks living over $\overline{\cR}_{d+1}$.  Note that a fiber is never compact; every disk---nodal or not---has boundary punctures.

Finally, we let $\overline{\cS}_{d+1}^\circ \subset \overline{\cS}_{d+1}$ denote the open subspace obtained by removing the nodal points of each fiber.

For any $r \in \overline{\cR}_{d+1}$, we let $\cS_r \subset \overline{\cS}_{d+1}^\circ$ denote the fiber above $r$.
\end{notation}

\begin{example}
If $d=2$, then $\overline{\cR}_{2+1}$ is homeomorphic to a single point. $\overline{\cS}_{2+1}$ is homeomorphic to a closed disk with three boundary points missing, as is $\overline{\cS}_{2+1}^\circ$.

If $d=3$, then $\overline{\cR}_{3+1}$ may be identified a closed unit interval $[0,1]$. The universal family $\overline{\cS}_{3+1} \to [0,1]$ has the property that the fiber over any element of the open interval $(0,1)$ is homeomorphic to a closed disk minus four boundary points. Over either endpoint---$0$ or $1$---the fiber is a wedge sum of two disks with two boundary points missing on each disk; in each fiber, the wedge point is the nodal point. Finally, the space $\overline{\cS}_{3+1}^{\circ}$ is obtained by removing exactly two points (the nodal points---one nodal point from each boundary element of $[0,1]$) from $\overline{\cS}_{3+1}$.

More generally, $\overline{\cS}_{d+1}^{\circ}$ is obtained from $\overline{\cS}_{d+1}$ by removing $i$ wedge points (i.e., $i$ nodal points) from every fiber living over a codimension $i$ stratum of $\overline{\cR}_{d+1}$.
\end{example}

\begin{choice}[Strip-like ends $\epsilon$]\label{choice:strip like ends}
We assume we have chosen strip-like ends near the nodes and boundary marked points of each fiber of $\overline{\cS}_{d+1}^\circ \to \overline{\cR}_{d+1}$. See Sections~(8d), (9a), and~(9c) of~\cite{seidel-book}.

We denote these strip like ends $\epsilon$ when necessary.

We assume we have also chosen diffeomorphisms $|\Delta^1| \cong [0,1]$ once and for all, so that the strip like ends are biholomorphic embeddings
	\eqnn
	\epsilon: [0,\infty) \times |\Delta^1|  \to \cS_r
	\qquad
	\text{or}
	\qquad
	\epsilon: (-\infty, 0] \times |\Delta^1| \to \cS_r.
	\eqnd
We denote by $\epsilon_i$ the strip-like end at the $i$th puncture.
\end{choice}

\begin{notation}[$\circ_i$]
Recall that every codimension one stratum of $\overline{\cR}_{d+1}$ can be written as a direct product $\overline{\cR}_{d_2 +1} \times \overline{\cR}_{d_1 +1}$; indeed, for a given $d_1$, and for any $1 \leq i \leq d_1$, there is an $i$th wedging map
	\eqn\label{eqn. R-gluing}
	\circ_i:
	\overline{\cR}_{d_2+1} \times \overline{\cR}_{d_1 +1} \to \overline{\cR}_{d+1},
	\qquad
	d_2 + d_1 - 1 = d
	\eqnd
which glues the 0th boundary vertex of a disk with $d_2+1$ marked points to the $i$th boundary vertex of a disk with $d_1 + 1$ marked points. For $r_1 \in \overline{\cR}_{d_1 +1}$ and $r_2 \in \overline{\cR}_{d_2+1}$, we let $r_2 \circ_i r_1$ denote the image.

We will also write $\cS_{r_2} \circ_i \cS_{r_1}$ for the corresponding nodal disk.

\end{notation}

Finally, because of our choice of strip-like ends, we can parametrize the corners of $\overline{\cR}_{d+1}^\circ$; for instance, in codimension one, the maps from~\eqref{eqn. R-gluing} extend to maps
	\eqn\label{eqn. R-corner parametrization}
	\overline{\cR}_{d_2+1}^\circ \times \overline{\cR}_{d_1 +1}^\circ \times [0, \epsilon)\to \overline{\cR}_{d+1}^\circ,
	\qquad
	d_2 + d_1 - 1 = d	
	\eqnd
(the dependence on $1 \leq i \leq d_1$ is suppressed in the above notation).
See also Sections~(9e) and~(9f) of~\cite{seidel-book}.

Our main interest is in a lift of \eqref{eqn. R-corner parametrization}:

\begin{notation}[$\#_i$ and $\#_{i,\tau}$]
Let $\overline{\sS}_{d_2,d_1} \to \overline{\cR}_{d_2 + 1} \times \overline{\cR}_{d_1 + 1} \times [0,\epsilon)$ denote the map obtained by pulling back $\overline{\cS}_{d_1+1}$ and $\overline{\cS}_{d_2+1}$ along the projections to $\overline{\cR}_{d_1+1}$ and $\overline{\cR}_{d_2+1}$, then taking the coproduct of these two pullbacks. Concretely, a fiber of $\overline{\sS}_{d_2,d_1}$ over $(r_2,r_1,\tau)$ is the disjoint union $\cS_{r_2} \coprod \cS_{r_1}$.  Then the gluing operation induced by the strip-like ends defines a map
	\eqn\label{eqn. S-corner parametrization}
	\#_i: \overline{\sS}_{d_2,d_1} \times [0,\epsilon) \to \overline{\cS}_{d+1}^\circ,
	\qquad
	d_2 + d_1 - 1 = d	
	\eqnd
where the restriction of $\#_i$ to time $\tau \in [0,\epsilon)$ will be denoted by $\#_{i,\tau}$.
\end{notation}

\begin{remark}
Note the font $\overline{\sS}_{d_2,d_1}$ rather than $\overline{\cS}_{d_2,d_1}$; we use the former because Savelyev uses the latter font to indicate a different entity in~\cite{savelyev}.
\end{remark}

\begin{notation}[$\#_{i,\tau}$]
Let us describe $\#_{i,\tau}$ for the sake of establishing further notation. Fix $\tau \in [0,\epsilon)$ and elements $r_1 \in \overline{\cR}_{d_1+1}, r_2 \in \overline{\cR}_{d_2+1}$. Having fixed our strip-like ends long ago, the $i$th gluing operation identifies two open subsets of $\cS_{r_1}$ and $\cS_{r_2}$ to obtain a new disk $\cS_{r}$. The strip-like ends endow $\cS_r$ with a thick-thin decomposition, where we can holomorphically identify the ``thin'' region of $\cS_r$ with $(-\tau,\tau) \times |\Delta^1|$, and these thin regions are precisely the regions where the gluing operation has non-singleton fibers (i.e., this is the region over which $\cS_{r_1}$ and $\cS_{r_2}$ are glued). The gluing maps
	\eqnn
	\cS_{r_2} \coprod \cS_{r_1} \times \{\tau\} \to \cS_r
	\eqnd
(where $r$ depends on $\tau$) define the maps $\#_{i,\tau}$. When $\tau = 0$, we have a map
	\eqn\label{eqn. sharp 0}
	\#_{i,0} : \cS_{r_2} \coprod \cS_{r_1} \to \cS_{r_2} \circ_i \cS_{r_1}.
	\eqnd
\end{notation}

\subsection{Simplices and inserting posets}

Now let us consider the simplicial analogue of the previous section's constructions.

\begin{notation}
Fix an integer $d \geq 0$. We let $[d]$ denote the linear poset given by
	\eqnn
	[d] = \{0 < 1 < \ldots < d\}.
	\eqnd
It is the unique linear order with $d+1$ elements (up to unique choice of order-preserving isomorphism).
\end{notation}

\begin{notation}[$A_2 \circ_i A_1$]
Let $A_2$ and $A_1$ be finite, non-empty, linearly ordered posets. We let $d_1 = \#A_1 - 1$. Then for any $1 \leq i \leq d_1$, we can construct a new linear poset
	\eqnn
	A_2 \circ_i A_1
	\eqnd
by gluing $A_2$ into $A_1$ as follows: Identify $\min A_2$ with the $(i-1)$st element of $A_1$, and identify $\max A_2$ with the $i$th element of $A_1$. We see that $A_2 \circ_i A_1 \cong [d_2 + d_1 - 1]$ as posets.
\end{notation}

We have a natural gluing map of posets
	\eqn\label{eqn. combinatorial sharp}
	\#_i: A_2 \coprod A_1 \to A_2 \circ_i A_1.
	\eqnd
By taking the geometric realization, we obtain a continuous map of topological spaces
	\eqnn
	|\Delta^{A_2}| \coprod |\Delta^{A_1}| \to |\Delta^{A_2 \circ_i A_1}|
	\eqnd
which, by (canonically) identifying $A_2 \cong [d_2]$ and $A_1 \cong [d_1]$ as linear posets, is equivalent to a continuous map
	\eqn\label{eqn. sharp simplices}
	\#_i: |\Delta^{d_2}| \coprod |\Delta^{d_1}| \to |\Delta^{d_2 + d_1 - 1}|.
	\eqnd
Concretely, $\#_i$ simplicially includes $|\Delta^{d_2}|$ and $|\Delta^{d_1}|$ as subsimplices of $|\Delta^{d_2 + d_1-1}|$, and these inclusions overlap along the edge between the $i$th and $(d_2 + i)$th vertices of $|\Delta^{d_2 + d_1 - 1}|$.

\begin{remark}
\eqref{eqn. sharp simplices} should be compared with the map~\eqref{eqn. sharp 0} from Section~\ref{section. curves operad}.
\end{remark}

\subsection{Operadic compatibility}

We start with fixing the following notation to avoid confusion arising from two
integer-valued indices that will be used.

\begin{notation}[The indices $n$ and $d$]
Fix a Liouville bundle  $E\to B$. We eventually want to define an $A_\infty$-category $\cO_j$ associated to any smooth map $j: |\Delta^n| \to  B $; in particular, we must define the $A_\infty$-operations $\mu^d$ for the $A_\infty$-categories $\cO_j$. In this section, the integers $n$ and $d$ will be used precisely for these purposes.
\end{notation}

Suppose we are given a simplicial map $\beta: |\Delta^d| \to |\Delta^n|$. (This is induced by a function $[d] \to [n]$, but this function need not be order-preserving.)\footnote{
Savelyev uses the notation $u(m_1,\ldots,m_d,n)$, but the data of $m_1,\ldots,m_d,n$ is equivalent to the data of a single simplicial map $\beta: |\Delta^d| \to |\Delta^n|$, or equivalently, a map of sets $\beta: [d] \to [n]$.
} We seek smooth maps
	\eqnn
	\nu_{\beta}: \overline{\cS}_{d+1}^\circ \to |\Delta^n|
	\eqnd
satisfying the following properties. (See the first row of Figure~\ref{figure. u_beta}.)

\newenvironment{natural-system}{
	  \renewcommand*{\theenumi}{(NS\arabic{enumi})}
	  \renewcommand*{\labelenumi}{(NS\arabic{enumi})}
	  \enumerate
	}{
	  \endenumerate
}

\begin{natural-system}
	\item\label{NS. strip condition u} Fix any $0 \leq k \leq d$. For any $r \in \overline{\cR}_{d+1}$, consider the fiber $\cS_r \subset \overline{\cS}_{d+1}^\circ$. Then for the edge from $k-1$ to $k$ in $|\Delta^d|$, the diagram
		\eqnn
		\xymatrix{
		[0,\infty) \times |\Delta^1| \ar[dr] \ar[r]^-{\epsilon_k}
			& \cS_r \subset \overline{\cS}_{d+1}^\circ \ar[r]^-{\nu_{\beta}}
			& |\Delta^n| \\
			& |\Delta^1| \ar[ur]_{\beta_{k-1,k}}
		}
		\eqnd
	commutes. (Here, $\beta_{k-1,k}$ is the simplicial inclusion of the edge from the $(k-1)$st vertex to the $k$th.) In plain English, this means that $\nu_{\beta}$ is compatible with the strip-like end parametrization by $|\Delta^1|$ near the $k$th puncture of $\cS_r$.
	
	Note we are using Choice~\ref{choice:strip like ends}; also, when $k=0$, the domain of $\epsilon_k$ should be $(-\infty,0] \times |\Delta^1|$ as opposed to $[0,\infty) \times |\Delta^1|$.
	\item\label{NS. vertex condition u} Now consider the boundary of $\cS_r$, and remove the images of the strip-like ends from this boundary. This results in $d+1$ disconnected open intervals, and we enumerate them so that the $(k-1)$st interval is contained in the boundary arc of $\cS_r$ beginning at the $(k-1)$st marked point and ending at the $k$th marked point.
	
	We demand that $\nu_{\beta}$ sends all of the $k$th interval to the vertex $\beta(k)$.
\end{natural-system}

    \begin{figure}
    \eqnn
    			\xy
    			\xyimport(8,8)(0,0){\includegraphics[width=3in]{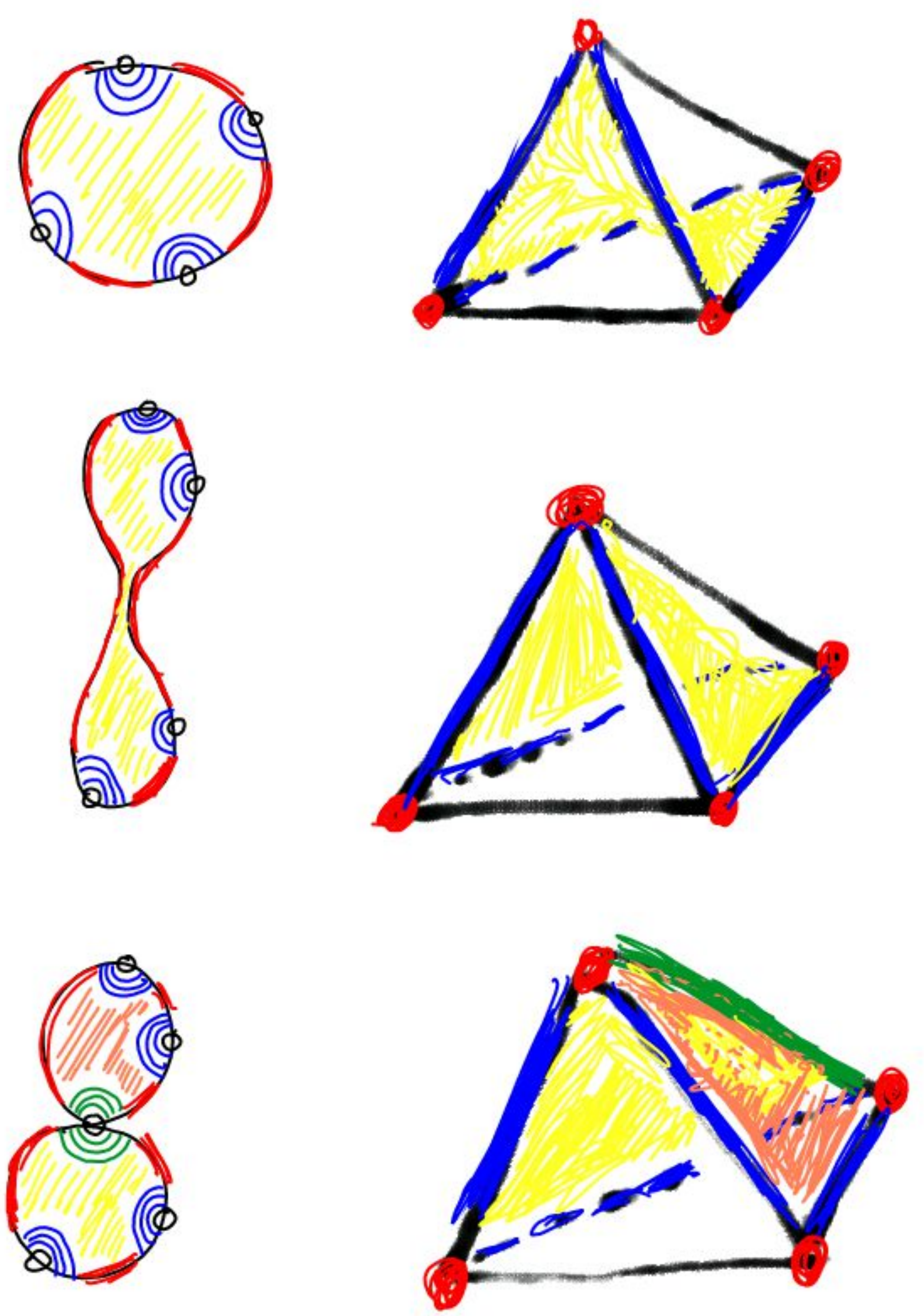}}
					%First row of images
					,(1,6.2)*{L_0}
					,(0,7.2)*{L_1}
					,(2,7.6)*{L_2}
					,(2.5,6.5)*{L_3}
					,(3.5,6)*{0}
					,(5.2,8)*{1}
					,(6.5,6)*{2}
					,(7.4,7)*{3}
					%Second row of images
					,(1.5,3.1)*{L_0}
					,(0.8,4.3)*{L_1}
					,(1.7,5.5)*{L_2}
					,(1.5,4.5)*{L_3}
					,(3.1,3)*{0}
					,(4.7,5)*{1}
					,(6.5,3)*{2}
					,(7.5,4)*{3}
					%Third row of images
					,(1,0.15)*{L_0}
					,(-0.1,0.7)*{L_1}
					,(0.2,1.8)*{L_1}
					,(1.6,2)*{L_2}
					,(1.3,1.2)*{L_3}
					,(3.5,0.2)*{0}
					,(4.9,2.3)*{1}
					,(7.3,0.3)*{2}
					,(8,1.4)*{3}
					%Labels
					,(0.8,-0.2)*{\cS_r}
					,(5.5,-0.3)*{|\Delta^3|}
    			\endxy
    \eqnd
    \begin{figurelabel}
    \label{figure. u_beta}
    An image of $\nu_{\beta}$, restricted to $\cS_r$ for various $r \in \cR_{3+1}$, when $\beta$ is the identity $\beta: |\Delta^3| \to |\Delta^3|$. The blue arcs on the left-hand disk are images of the 1-simplices $\{t\} \times |\Delta^1| \subset (0,\infty) \times |\Delta^1|$ under strip-like parametrizations; they are sent to the blue edges indicated on the right-hand 3-simplices. The red thick edges on the disks are the ``open intervals'' referred to in the main text (in practice, these thick edges are labeled by Lagrangians $L_i$); these red edges are collapsed to the vertices labeled in red on the 3-simplices. (The edge labeled by $L_i$ is sent to the $i$th vertex.) In yellow is a drawing of the image of $\nu_{\beta}$ in the 3-simplex. In the bottom most image, the two components of a nodal disk are labeled by orange and yellow, and these are sent to two faces of the 3-simplex as indicated. Note the new green strip-like ends, and the newly highlighted green edge of the 3-simplex.
    \end{figurelabel}
    \end{figure}

At this point, we see that as $r$ approaches a boundary stratum of $\overline{\cR}_{d+1}$ (i.e., as disks degenerate), we would like the restrictions $\nu_{\beta}|_{\cS_r}$ to behave in a way compatible with the boundary faces of the $n$-simplex. See Figure~\ref{figure. u_beta}. We make this compatibility (which Savelyev refers to as ``natural'' in~\cite{savelyev}) precise:

\begin{notation}[$\star_i$]
Fix $d_2, d_1 \geq 2$ such that $d = d_2 + d_1 - 1$. Choose $1 \leq i \leq d_1$.

Then our fixed map $\beta: |\Delta^d| \to |\Delta^n|$ induces maps
	$\beta_1: |\Delta^{d_1}| \to |\Delta^n|$
	and
	$\beta_2: |\Delta^{d_2}| \to |\Delta^n|$
so that the diagram
	\eqnn
	\xymatrix{
	 | \Delta^{d_2}| \coprod |\Delta^{d_1}| \ar[dr]_{\beta_2 \coprod \beta_1} \ar[rr]^-{\#_i}
	 	&& |\Delta^{d}| \ar[dl]^{\beta} \\
		& |\Delta^n|
	}
	\eqnd
commutes. (Here, $\#_i$ is the map from~\eqref{eqn. sharp simplices}.)

On the other hand, if we are given maps $\nu_{\beta_1}: \overline{\cS}_{d_1 + 1}^\circ \to |\Delta^n|$ and  $\nu_{\beta_2}: \overline{\cS}_{d_2 + 1}^\circ \to |\Delta^n|$, the conditions \ref{NS. strip condition u} and  \ref{NS. vertex condition u} guarantee the existence of a (unique) map making the following diagram commute:
	\eqnn
		\xymatrix{
		\overline{\cS}_{d_2 + 1}^\circ \coprod \overline{\cS}_{d_1 +1}^\circ \ar[rr]^-{\#_{i,\tau=0}}
			\ar[dr]_{\nu_{\beta_2} \coprod \nu_{\beta_1}}
			&& \overline{\cS}_d^\circ |_{\overline{\cR}_{d_2+1} \circ_i \overline{\cR}_{d_1+1}}
				\ar[dl]^{\exists!} \\
		& |\Delta^{n}|
		}
	\eqnd
Here, the notation $\overline{\cS}_{d+1}^\circ |_{\overline{\cR}_{d_2+1} \circ_i \overline{\cR}_{d_1 +1}^\circ}$ denotes the family $\overline{\cS}_{d+1}^\circ$ restricted to the image of the map $\circ_i$ from~\eqref{eqn. R-gluing}.

Extending the gluing parameter $\tau$ from 0 to an element of $[0,\epsilon)$, there is a  neighborhood $U_{d_2,d_1,i} \supset \overline{\cR}_{d_2+1} \circ_i \overline{\cR}_{d_1+1}$ such that there is a unique extension $\nu_{\beta_2} \star_i \nu_{\beta_1}$ making the diagram below commute:
	\eqnn
		\xymatrix{
		\overline{\cS}_{d_2 + 1}^\circ \coprod \overline{\cS}_{d_1 +1}^\circ \times [0,\epsilon) \ar[rr]^-{\#_i}
			\ar[dr]_{\nu_{\beta_2} \coprod \nu_{\beta_1}}
			&& \overline{\cS}_d^\circ |_{U_{d_2,d_1,i}}
				\ar[dl]^{\nu_{\beta_2} \star_i \nu_{\beta_1}}  \\
		& |\Delta^{n}|.
		}
	\eqnd
Explicitly, on the thin strips $|\Delta^1| \times [-\tau,\tau]$, we declare $\nu_{\beta_2} \star_i \nu_{\beta_1}$ to equal the composition of the projection to $|\Delta^1|$ with the simplicial inclusion of the edge from the $(i-1)$st vertex to the $i$th vertex.
\end{notation}

Hence it is natural to demand the following:

	\begin{natural-system}
	\setcounter{enumi}{2}
	\item\label{NS. star compatible} For all $d_1, d_2, i$ as above, we demand that $\nu_{\beta}$ agrees with $\nu_{\beta_2} \star_i \nu_{\beta_1}$ on some neighborhood of $\overline{\cR}_{d_2+1} \circ_i \overline{\cR}_{d_1+1}$. That is,
		\eqnn
		\nu_{\beta} = \nu_{\beta_2} \star_i \nu_{\beta_1}
		\qquad
		\text{on}
		\qquad
		\overline{\cS}_{d+1}^\circ|_{U_{d_2,d_1,i}}
		\eqnd
	(possibly after replacing $U_{d_2,d_1,i} \supset \overline{\cR}_{d_2+1} \circ_i \overline{\cR}_{d_1+1}$ with some other neighborhood containing $\overline{\cR}_{d_2+1} \circ_i \overline{\cR}_{d_1+1}$, if necessary).
	\end{natural-system}
Finally, while we have fixed $n$ up until now, we demand that $\nu_{\beta}$ is functorial as the codomain of $\beta$ varies:
	\begin{natural-system}
	\setcounter{enumi}{3}
	\item\label{NS. j functorial} Let $\alpha: [n] \to [n']$ be a map of posets; by abuse of notation, we also denote the induced simplicial map $\alpha: |\Delta^n| \to |\Delta^{n'}|$. Then for any $\beta: |\Delta^d| \to |\Delta^n|$, we demand
		\eqnn
		\alpha \circ \nu_{\beta} = \nu_{\alpha \circ \beta}.
		\eqnd
	\end{natural-system}
	
\begin{defn}[Natural system] \label{defn. natural system}
For every $n,d \geq 0$ and every simplicial map $\beta: |\Delta^d| \to |\Delta^n|$, choose a smooth map
	\eqnn
	\nu_{\beta}: \overline{\cS}_{d+1}^\circ \to |\Delta^n|.
	\eqnd
The collection $\{\nu_{\beta}\}$ is called a {\em natural system} if \ref{NS. strip condition u}, \ref{NS. vertex condition u}, \ref{NS. star compatible}, and \ref{NS. j functorial} above are satisfied.
\end{defn}

A standard inductive argument shows the following:

\begin{prop}[Proposition~3.4 of~\cite{savelyev}]\label{prop:natural systems exist}
Natural systems exist.
\end{prop}

\begin{choice}\label{choice. natural system}
We will choose a natural system $\{\nu_{\beta}\}$ once and for all. (Note this is independent of any symplectic geometry or of any choice of Liouville bundle $E \to B$.)
\end{choice}

\begin{remark}\label{remark:simplices are homeo to families of disks}
One can prove that, given a natural system, each of the maps $\nu_{\beta}: \overline{\cS}_{n+1}^\circ \to |\Delta^n|$ (when $\beta = \id$) is a degree one map on the interior; one can roughly think of $\nu_{\beta}$, then, as ``homeomorphisms on the interior.'' The naturality of the system says that these topological equivalences can be chosen in such a way that the gluing operations of disks are compatible with the insertion operation of simplices.
\end{remark}

\subsection{Collars on boundaries of simplices}
We conclude this section with a final choice, made once and for all for every standard simplex $|\Delta^d| \subset \RR^{d+1}$.

\begin{choice}[Collars of simplices]\label{choice. simplicial collars}
For every closed, codimension one face $F \subset |\Delta^d|$, we choose a small open neighborhood $U_F \subset |\Delta^d|$ together with a smooth retraction $\pi_F: U_F \to F$, which one thinks of as a projection map.

We choose the data of $(U_F, \pi_F)$ such that the following holds:

\enum
    \item (The neighborhoods are mutually small.)
    If $A \subset |\Delta^d|$ is a closed subsimplex, let
    	\eqnn
    	\cU_A := \bigcap_{A \subset F} U_F.
    	\eqnd
    If $A, A' \subset |\Delta^d|$ are two closed subsimplices such that $A \cap A' = \emptyset$, we demand that
    	\eqn\label{eqn. collar intersections are empty}
    	\cU_A \cap \cU_{A'} = \emptyset.
    	\eqnd
	\item (Neighborhoods have controlled intersections.)
	Let $F$ and $F'$ be two codimension one faces with intersection $G = F \cap F'$. Then
		\eqn\label{eqn. collars controlled}
		U_F \cap F' \subset U_G.
		\eqnd
\enumd
\end{choice}

\begin{remark}
Let us motivate the conditions. First, the collars are chosen so that we may trivialize certain data over the collars. (See for example \ref{item. floer data smooth} of Choice~\ref{choice. floer data}.)

The intersection property~\eqref{eqn. collar intersections are empty} guarantees that these local trivializations do not enforce a global trivialization.

The intersection constraint~\eqref{eqn. collars controlled} enables choices that are made inductively by dimension. For example, if we have already chosen data on lower-dimensional face $F'$ for which a trivialization does not extend beyond $U_G$, $U_F$ must be sufficiently small for us to be able to trivialize on $U_F$.
\end{remark}

\begin{remark}
Note that these collars are independent of our choice of natural systems. The collar choices are likewise independent of any symplectic geometry. They obviously exist by a simple inductive argument on dimension.
\end{remark}

\clearpage

\section{Compactness}\label{section. compactness  estimates}

In this section, we study compactness properties of the moduli space
of holomorphic sections of a Liouville bundle $\pi:P \to \Sigma$ over
a surface $\Sigma = D^2 \setminus \{z_0,\ldots, z_k\}$. There are two fundamental
analytical ingredients to establish: $C^0$-estimates and
 energy estimates. These estimates imply Gromov compactness as usual, and establish both the $A_\infty$ relations for $\cO_j$, and the existence and properties of continuation maps (which we study in the next section).

\begin{remark}
We reiterate that we are using {\em cohomological} conventions for our Floer complexes.
So for example, given a morphism $q$ from a brane $L$ to another brane $L'$,
the differential $\mu^1 q$ is computed by counting solutions to the equation
	\eqn\label{eq:CR-boundary}
\begin{cases}
\delbar_J u = 0 \\
u(-\infty) = p, \, u(\infty) = q \\
u(\tau,0) \in L, \, u(\tau,1) \in L'.
\end{cases}
	\eqnd
This is the same convention as in \cite{seidel-long-exact-sequence}.
\end{remark}

\subsection{Curvature}

\begin{notation}[$\Omega$]
Fix a Liouville bundle $E \to B$ with base $B$, and fix a Liouville form $\Theta \in \Omega^1(E)$ as in Remark~\ref{remark. can choose liouville form}; we have the associated two-form
	\eqnn
	\Omega=d\Theta
	\eqnd
on $E$.
\end{notation}

Suppose that $B = \Sigma$ is a Riemann surface. Then the curvature of the connection associated to $\Theta$ (Definition \ref{defn. theta connection}) can be regarded as a 2-form on $\Sigma$ with values in smooth functions of the fibers. More precisely:

\begin{prop}\label{prop. curvature controlled by f}
For any orientation-respecting volume form $\omega_\Sigma$ on $\Sigma$, we have
	\eqn\label{eq:Omega-f-beta}
	\Omega|_{HTE} = f \pi^* \omega_{\Sigma}|_{HTE}
	\eqnd
where $f: E \to \RR$ is a smooth function.
\end{prop}

For proofs, see Theorem 4.2.9 of \cite{oh:book1} or (1.4) of~\cite{seidel-long-exact-sequence}.

\begin{defn}[Compare with p.1007 \cite{seidel-long-exact-sequence}]\label{defn:+curvature}
We say that the curvature of $\Theta$ is {\em non-negative} if $\Omega|_{HTE}$ is non-negative for the orientation on $HTE$ (induced by that of $\Sigma$).

More generally, we say $\Theta$ is \emph{$(-C)$-pinched (from below)} if we have that
	$$
\inf_{x \in E} f(x) \geq - C
	$$
for some $C \geq 0$.
(Here, the function $f$ is as in \eqref{eq:Omega-f-beta}.)
\end{defn}

\begin{example}
$\Theta$ has nonnegative
curvature if and only if it is $0$-pinched.
\end{example}

\begin{remark}
In all our choices of $\Theta$, we can arrange for $\Theta$ to be $(-C)$-pinched. This is for two reasons: First, $\Theta$ has behavior controlled outside a set $K \subset E$ that is proper over $\Sigma$ (by assumption, $K$ is a set whose complement is a fiberwise cylindrical region). Second, on a strip-like end of $\Sigma$, $\Theta$ is trivialized in the $\tau$ direction (e.g., in the $(-\infty,0]$ direction of the strip $(\infty,0] \times [0,1]$)), so one may extend $K$ over a compactification of $\Sigma$ (e.g., by filling in the punctures of $\Sigma$ with chords).
\end{remark}

\subsection{On the interior of disks}

We now establish that holomorphic curves with certain Lagrangian boundary conditions must be contained in some compact region of a Liouville bundle. A key result is Proposition~\ref{prop:Deltarv>0}, which shows that---if a holomorphic curve $u: \Sigma \to E$ intersects the cylindrical region of $E$---the natural conical coordinate $r= e^s$ is a subharmonic function on $\Sigma$.

\begin{remark}
Recall that subharmonic functions behave like ``convex up'' functions, in that non-constant maxima are attained only along the boundary of the domain. Thus, knowing that the cylindrical coordinate of $u$ is constrained along the interior of $\Sigma$, by imposing appropriate boundary conditions on our Lagrangian family (to obtain constraints on the behavior of $u$ along the boundary of $\Sigma$), we obtain the desired $C^0$ estimate in Theorem~\ref{thm:C0estimate+CL}.
\end{remark}

\begin{lemma}\label{lem:Deltarv=vOmega}
Fix an conical-near-infinity conical almost-complex structure on $E \to B$ (Definition~\ref{defn. J conical near infinity}). Let $v: \Sigma \to E$ be any $(j,J)$-holomorphic section. Then there exists some subset $K \subset E$, proper over $\Sigma$, outside of which we have
	\eqn\label{eq:Deltarv=vOmega}
\Delta (r\circ v) \omega_\Sigma = v^*\Omega.
	\eqnd
\end{lemma}

\begin{proof}
We take $K$ to be the same set as in the definition of conical-near-infinity almost-complex structure (Definition~\ref{defn. J conical near infinity}). We may then choose a global $r$ coordinate.\footnote{For example, by constructing an appropriate bundle of Liouville domains $M_\Sigma \to \Sigma$, embedding $M_\Sigma \into E$ over $\Sigma$, then defining $r$ by the Liouville flow time of the fiberwise contact boundaries.}
Using $\Theta \circ J = dr$ and $J\circ dv = dv \circ j$, we have
	$$
d(r\circ v) = v^*dr = v^*(\Theta \circ J) = \Theta (J\circ dv)
= \Theta (dv \circ j) = v^*\Theta \circ j.
	$$
Therefore we have
	$$
\Delta (r\circ v)\omega_\Sigma = - d(d(r\circ v)\circ j) = d(v^*\Theta) = v^*\Omega.
	$$
This finishes the proof.
\end{proof}

\begin{notation}[$l$ and $\omega_\beta$]\label{notation. ell omega beta}
Now let $\beta$ be a positive 2-form on $\Sigma$---this means that
	\eqnn
	\beta = l \omega_\Sigma
	\eqnd
for some positive function $l: \Sigma \to \RR$. We let
	\eqnn
	\Omega_\beta = \Omega + \pi^*\beta.
	\eqnd
\end{notation}

\begin{remark}
If $\ell: \Sigma \to \RR$ is sufficiently positive, then $\Omega_\beta$ is a symplectic form. Moreover, given a conical-near-infinity choice of almost complex structure $J$, choosing $\ell$ sufficiently positive makes the projection $E \to B$ a $(J,j)$-holomorphic map (Lemma~2.1 of~\cite{seidel-long-exact-sequence}).
\end{remark}

Consider the compatible metric
	$$
g(X,Y) := \frac12(\Omega_\beta(X, JY) + \Omega_\beta(Y, JX).
	$$
As usual, we have the identity
	$$
|dv|^2 = |\del_J v|^2 + |\delbar_J v|^2,  \quad 2 v^*\Omega_\beta = (|\del_J v|^2 - |\delbar_J v|^2)\, \omega_\Sigma
	$$
for \emph{any smooth} section $v$.  We have more when $v$ is holomorphic:

\begin{prop}\label{prop:Deltarv>0} Suppose $v:\Sigma \to E$ is a $(j,J)$-holomorphic section. Then
	$$
\Delta(r \circ v) = \frac12|(dv)^{\text{\rm v}}|^2 + f(v) \ell
	$$
where $f$ is the same function from~\eqref{eq:Omega-f-beta} and $\ell$ is from Notation~\ref{notation. ell omega beta}.

In particular, if the pull-back connection of $v^*E$ has nonnegative curvature, i.e., if $f(v) \geq 0$, and if $\ell$ is large enough,
then $r \circ v$ is a subharmonic function on $v^{-1}(E \setminus K)$. (Here, $K$ is the same set as in Lemma~\ref{lem:Deltarv=vOmega}.)
\end{prop}

\begin{proof}
If $\delbar_J v = 0$, then we obtain
	\eqn\label{eq:vOmegabeta}
v^*\Omega_\beta = \frac12 |dv|^2\, \omega_\Sigma.
	\eqnd
Splitting $dv = (dv)^{\text{\rm v}} + (dv)^{\text{\rm h}}$ into horizontal and vertical components of $dv$,
we compute
	$$
\frac12|(dv)^{\text{\rm h}}|^2\, \omega_\Sigma = (f(v) + 1)\beta
	$$
where $f$ is the function given by \eqref{eq:Omega-f-beta}.
Rewriting \eqref{eq:vOmegabeta} as
	\eqn\label{eq:v*Omega+beta}
v^*\Omega + v^*\pi^*\beta = \frac12|(dv)^{\text{\rm v}}|^2 \omega_\Sigma + (f(v)+1) \beta
	\eqnd
and using $\pi \circ v = \id_\Sigma$, we obtain
    \eqn\label{eqn. v*Omega computation}
    v^*\Omega = \frac12|(dv)^{\text{\rm v}}|^2 \omega_\Sigma + (f(v) + 1) \beta - (\pi \circ v)^*\beta
    = \frac12|(dv)^{\text{\rm v}}|^2 \omega_\Sigma + \ell f(v) \omega_\Sigma.
	\eqnd

Now combine Lemma~\ref{lem:Deltarv=vOmega} with~\eqref{eqn. v*Omega computation}.
\end{proof}

\subsection{Boundary conditions}
The definitions below simply give names to the conditions we can guarantee in our set-up.

\begin{defn}\label{defn. E translation invariant}
Let $E \to \Sigma$ be a Liouville bundle, and fix strip-like ends of $\Sigma$. We say that $E$ is {\em translation-invariant over the strip-like ends} if, over the strip-like ends, $E$ is equipped with an isomorphism to a pullback bundle
	\eqnn
	\xymatrix{
	p^*E_{[0,1]}  \ar[r] \ar[d] & E_{[0,1]} \ar[d] \\
	R \times [0,1] \ar[r]^p & [0,1]
	}
	\eqnd
where $R$ is either $(-\infty,0]$ (for incoming strips) or $[0,\infty)$ (for outgoing strips).
\end{defn}

The following is a variation on Section 2.1 of Seidel's work~\cite{seidel-long-exact-sequence}.

\begin{defn}[Lagrangian boundary conditions for bundles]\label{defn. lagrangian boundary condition}
Fix a Liouville bundle $E \to \Sigma$ and strip-like ends on $\Sigma$. Assume $E$ is translation-invariant over the strip-like ends (Definition~\ref{defn. E translation invariant}).
A Lagrangian boundary condition {\em suitable for our purposes}  is an $(n+1)$-dimensional submanifold
$\cL \subset E|_{\del \Sigma}$ equipped with the data of a smooth function $K_{\cL}: \cL  \to \RR$,
 called a Liouville primitive, such that
\begin{enumerate}
\item $\cL$ is fiberwise conical near infinity,
\item $\pi|_\cL: \cL \to \del \Sigma$ is a submersion,
\item $\Theta|_{\cL} = dK_\cL$, and
\item $K_\cL$ is fiberwise (affine) linear outside a subset $K \subset E$ which is proper over $\Sigma$, and
\item $\cL$ is translation-invariant on the strip-like ends. (For example, for a strip-like end modeled on $(-\infty,0]$, $\cL$ can be obtained by pulling back a pair of branes $L_0 \subset E_0, L_1 \subset E_1$, along the projection $(-\infty,0] \times [0,1] \to [0,1]$.)
\end{enumerate}
\end{defn}

\begin{remark}
The conditions of Definition~\ref{defn. lagrangian boundary condition} imply that for every $z \in \del \Sigma$, the fiber $\cL_z$ is an exact Lagrangian submanifold of $E_z$, and
that $K_\cL|_{\cL_z}$ is  a Liouville primitive.
\end{remark}

The following is the fibration version of non-negative wrapping.

\begin{defn}\label{defn. non-negative lagrangian boundary condition}
Let $\cL$ be a boundary condition as in Definition~\ref{defn. lagrangian boundary condition}.
We call $\cL$ \emph{nonnegative (resp. nonpositive) relative to $\del \Sigma$}
if $\Theta(\xi) \geq 0$ (resp. $\Theta(\xi) \leq 0$)
for all $\xi \in T_x\cL$ whose projection to $T \del \Sigma$ compatible with the orientation of $\del \Sigma$.
\end{defn}

Finally, in our setting, one can arrange for the following:

\begin{defn}\label{defn. our kind of connection}
Let $\Sigma = D^2 \setminus \{z_0,\ldots,z_k\}$, and let $\cL \subset E \to \Sigma$ be a Lagrangian boundary condition as in Definition~\ref{defn. lagrangian boundary condition}. Choose also a strip-like end for every $z_i$, with $z_0$ incoming and others outgoing.

A connection induced by a Liouville form $\Theta$ (as in Remark~\ref{remark. can choose liouville form}) will be called {\em our kind of connection} if it is trivial on $\del \Sigma$ outside the strip-like ends, and if the connection is invariant under the translation on the strip-like ends.
\end{defn}

\begin{example}
For example, the conditions of Definition~\ref{defn. our kind of connection} hold on a strip-like end modeled after $(-\infty,0] \times [0,1]$ if $\Theta$ and $E$ are both pulled back along the projection $(-\infty,0] \times [0,1] \to [0,1]$.)
\end{example}

\begin{remark}
In all our examples, $E \to \Sigma$ is pulled back along a map $\Sigma \to |\Delta^n|$ which collapses strip-like ends to a single edge of $|\Delta^n|$, and collapses the rest of $\del \Sigma$ to vertices of $|\Delta^n|$; as such, the connections pulled back from a bundle $E' \to |\Delta^n|$ satisfy Definition~\ref{defn. our kind of connection}.
\end{remark}

\begin{remark}
If $\Theta$ induces our kind of connection (Definition~\ref{defn. our kind of connection}), and if for every connected component $c_i \subset \del \Sigma$, we choose some $x_i \in c_i$ and a brane $L_i \subset E_{x_i}$, one obtains a Lagrangian boundary condition $\cL$ by parallel transport along $\del \Sigma$. $\cL$ is non-negative in the sense of Definition~\ref{defn. non-negative lagrangian boundary condition}.
\end{remark}

\subsection{The \texorpdfstring{$C^0$}{C0} estimate}

\begin{theorem}\label{thm:C0estimate+CL}
Let $(\Sigma,j)$ be a Riemann surface $\Sigma  = D^2 \setminus \{z_0, \ldots, z_k\}$ where $\{z_0, \ldots, z_k\} \subset \del D^2$, and let $\pi: E \to \Sigma$ be a Liouville bundle with translation invariance (Definition~\ref{defn. E translation invariant}). Further choose:
\begin{itemize}
	\item A non-negative boundary condition $\cL$ as in Definitions~\ref{defn. lagrangian boundary condition} and~\ref{defn. non-negative lagrangian boundary condition}.
	\item A conical-near-infinity almost-complex structure $J$ on $E$ for which $E \to \Sigma$ is $(J,j)$-holomorphic,
	\item A Liouville form $\Theta \in \Omega^1(E)$ inducing our kind of connection $\Theta$ (Definition~\ref{defn. our kind of connection}) which is $(-C)$-pinched (Definition~\ref{defn:+curvature}),   and
	\item For each $i \in 1, \ldots, k$, a parallel transport chord $x_i$ from the $(i-1)$st boundary brane to the $i$th boundary brane, along with a parallel transport chord from $L_0$ to $L_k$.
\end{itemize}

Suppose that  $v: \Sigma \to E$ is a $(j,J)$-holomorphic section such that
\begin{itemize}
	\item $v(\del \Sigma) \subset \cL$, and
	\item $v$ converges to the parallel transport chords $x_i$ along the strip-like ends.
\end{itemize}

Then there exists some subset $A \subset E$, proper over $\Sigma$, and depending only on the $x_i$, such that
	$$
\Image v \subset A.
	$$
\end{theorem}

\begin{proof}
By Proposition~\ref{prop:Deltarv>0} we have
	$$
\Delta(r \circ v) \omega_\Sigma = \left({\frac12}|(dv)^{\text{\rm v}}|^2 + f(v)\right) \beta
	$$
wherever the image of $v$ is contained in
the cylindrical region of $E$. (That is, when $v$ has image outside of the set $K$ of the Proposition.)

We first consider the case of nonnegative curvature, i.e., $f \geq 0$.
 By the nonnegativity hypothesis, $r\circ v$ is a subharmonic function on $\Sigma$. Therefore it cannot have any local
maximum at an interior point and we have only to check its boundary behavior.

We compute the radial derivative $\frac{\del (r\circ v)}{\del \nu}$---i.e., the derivative along an outward pointing boundary vector:
	$$
\frac{\del (r\circ v)}{\del \nu}  =  dr \left(\frac{\del v}{\del \nu}\right)
= \Theta \circ J\left(\frac{\del v}{\del \nu}\right).
	$$
Since $v$ is $(j,J)$ holomorphic, we have
	$$
J\left(\frac{\del v}{\del \nu}\right) = J \circ dv\left(\frac{\del}{\del \nu}\right)
= dv \left(j\frac{\del}{\del \nu}\right).
	$$
Therefore if $\del / \del \tau$ is any positively oriented tangent vector along $\del \Sigma$, we have
	$$
\frac{\del(r\circ v)}{\del \nu} = - \Theta\left(\frac{\del v}{\del \tau}\right) \leq 0
	$$
by the nonnegativity assumption of $\cL$.

Therefore the strong maximum principle implies that $r\circ u$ cannot have
boundary local maximum anywhere on $\del \Sigma$.

For the $(-C)$-pinched case, we have $\inf f \ell \geq -C$  and so the function $r\circ v$ satisfies the differential inequality
	$$
\Delta (r\circ v) \geq -C, \quad \frac{\del(r\circ v)}{\del \nu} \leq 0.
	$$
At this stage, we can apply the standard elliptic estimates (see for example \cite[Theorem 3.7]{GT}).
Another more explicit way of proceeding is to
consider the function $g = (r\circ v) - \frac{C}2 t^2$ where $t: \Sigma \to \RR$ is the
pull-back function of the standard coordinates $(\tau,t)$ of $\RR \times [0, w]$ for some $w> 0$
via the \emph{slit domain representation} of the conformal structure of $\Sigma = D^2 \setminus \{z_0, \cdots, z_k\}$. (See \cite[Section 3.2]{bko:wrapped}, for example.)

Then $g$ is a subharmonic function. We can apply the strong maximum principle to the function $g$ to conclude
	$$
\sup_{z \in \Sigma^{\text{\rm end}}} g(z) \leq R_0
	$$
satisfying $\frac{\del(r\circ v)}{\del \nu} \leq 0$ since $\frac{\del t}{\del \nu} = 0$
along the boundary $\del \Sigma$.
 Therefore we conclude
	$$
\sup_{z \in \Sigma^{\text{\rm end}}} r \circ v(z) \leq R_0 + \frac{C}{2}.
	$$
This finishes the proof.
\end{proof}

\subsection{The energy estimate}

Fix a non-negative Lagrangian boundary condition $\cL$ (Definition~\ref{defn. lagrangian boundary condition} and~\ref{defn. non-negative lagrangian boundary condition}). We have that
	\eqn\label{eq:i*Theta}
\Theta|_{\cL} = dK_\cL + \pi^*(\kappa_\cL)
	\eqnd
for a one-form $\kappa_\cL \in \Omega^1(\del \Sigma)$ which vanishes on the strip-like
ends. (In our case, $\kappa_\cL = 0$ on all of $\del \Sigma$, but we include $\kappa_{\cL}$ in what follows for the interested reader.)

The action functional on the path space
	$$
\cP(L_0,L_1) = \{ \gamma \in C^\infty([0,1],M) \mid \gamma(0) \in L_0, \, \gamma(1) \in L_1\}
	$$
is given by
	$$
\cA_{L_0,L_1}(\gamma) = - \int \gamma^*\theta  + K_{L_1}(\gamma(1)) - K_{L_0}(\gamma(1)).
	$$
Using \eqref{eq:i*Theta}, we also obtain
	\eqn\label{eq:intv*Omega}
\int v^*\Omega = \sum_{e \in I^-} \cA_{L_0,L_1}(x_e) - \sum_{e \in I^+} \cA_{L_0,L_1}(x_e)
+ \int_{\del \Sigma} \kappa_\cL.
	\eqnd
On the other hand, we derive from \eqref{eq:v*Omega+beta}
	$$
\frac12|(dv)^{\text{\rm v}}|^2 = v^*\Omega - f( v) \beta
	$$
and hence
	$$
\frac12\int_\Sigma |(dv)^{\text{\rm v}}|^2 = \sum_{e \in I^-} \cA_{L_0,L_1}(x_e) - \sum_{e \in I^+} \cA_{L_0,L_1}(x_e)
+ \int_{\del \Sigma} \kappa_\cL - \int_\Sigma f(v) \beta.
	$$
Here we would like to mention that both the integrals $\int_{\del \Sigma} \kappa_Q$ and
$\int_\Sigma f(v)\, \beta$ are finite since $\kappa_Q = 0$ and $f(v) =0$ on the strip-like region
of $\Sigma$.  We also have
	$$
\frac12 \int_W |(dv)^{\text{\rm h}}|^2 = \int_W (f(v) + 1)\, \beta =  \int_W f(v)\, \beta  + \int_W \beta
	$$
for any compact domain $W \subset \Sigma$.

We summarize the above discussion into the following
uniform upper bound for the energy \emph{on any compact domain $W \subset \Sigma$} satisfying the property that $\kappa_Q = 0 = f(v)$ on $\Sigma \setminus W$.

\begin{prop}\label{vertE+f} Let $(\pi:E \to \Sigma,\Omega)$ and $\cL$ be as above and $W \subset \Sigma$ be
any given compact subdomain of $\Sigma$. Then
	\eqn
\frac12\int_W |dv|^2  = \int_\Sigma v^*\Omega + \int_W \beta
	\eqnd
for any $(j,J)$-holomorphic section $v: \Sigma \to E$.
\end{prop}

\begin{remark}
The reason why we restrict to compact domain $W \subset \Sigma$ is that
the form $\beta$ may not be integrable, unlike $f(v) \beta$.
Moreover, the integrals above depend on the section $v$ and may not be uniformly
bounded, mainly because the strip-like regions of $\Sigma = D^2 \setminus \{z_0, \ldots, z_k\}$
vary depending on the configuration of $\{z_0, \ldots, z_k\}$.
\end{remark}

\begin{remark}
By requiring translation invariance of $\omega_\Sigma$ on the strip-like ends of $\Sigma$, we conclude that the full integral $\int_\Sigma \omega_\Sigma$ is
infinite whenever there is at least one puncture on $\Sigma$. In choosing the 2-form $\beta = \ell \, \omega_\Sigma$, there are two competing interests:
\begin{itemize}
\item One one hand, we need the form $\Omega + \pi^*\beta$ to be nondegenerate,
\item On the other hand, we wish to make the form $\beta$ have finite integral over $\Sigma$.
\end{itemize}
In general we cannot achieve both wishes simultaneously. This is the reason why we
need to consider the horizontal energy on compact domains $W$ e.g., on $W = \Sigma \setminus \Sigma^{\text{\rm end}}$.
\end{remark}

\begin{remark}
However, because we are given a connection that is translation-invariant on the strip-like ends, when we restrict $v$ along a strip-like end, we may write
	$$
v(\tau,t) = (\tau,t,u(\tau,t))
	$$
where $u$ is a function satisfying
	$$
\frac{\del u}{\del \tau} + J \left(\dudt - X_{H}(\tau,t,u)\right) = 0.
	$$
This equation can be studied in the standard way of
classical Floer theory.
\end{remark}

Therefore with the uniform $C^0$-estimates at our disposal,
we can apply the Gromov-Floer type of compactness arguments to the moduli space of
pseudoholomorphic sections. (See \cite[Section 2.4]{seidel-long-exact-sequence} for some relevant details.)

\clearpage

\section{Non-wrapped Fukaya categories in Liouville bundles}\label{section. unwrapped}

The present section is occupied by  the construction of the $A_\infty$ category $\cO_j$ associated to
a simplex $j: |\Delta^n| \to B$. As usual we have fixed a Liouville bundle $E \to B$.

\subsection{Choice of objects}

\begin{choice}[$\LL_b$ and a partial ordering.]\label{choice. ordering on objects}
For every point $b \in B$, we choose a countable collection $\LL_b$ of eventually conical branes in the fiber $E_b$. We moreover choose a function
    \eqnn
    w: \LL_b \to \ZZ_{\geq 0} = \{n \in \ZZ , n \geq 0\}.
    \eqnd
We will often abbreviate a pair $(L_b, w(L_b))$ in the graph of $w$ by
    \eqnn
    L^{(w)}
    \eqnd
omitting $b$, and omitting the dependence of $w$ on $L_b$.

For a fixed $b \in B$, the collection $\{L_b\}$ is countable, so we may choose the function $w$ so that given two branes $L$ and $L'$, $L$ and $L'$ are either transverse (in the fiber), or $L=L'$ and $w=w'$.
\end{choice}

\begin{remark}\label{remark. cofinal sequences}
In~\cite{oh-tanaka-actions}, we will choose the ordering $w$ to encode cofinal sequences of non-negative wrappings of branes. We don't need these details in the present work, so we refer the reader to Section~2.2 of~\cite{oh-tanaka-actions} for more.
\end{remark}

\subsection{Choices of Floer data}

\begin{remark}\label{remark. transversality choices}
Because we have already chosen a favorite collection of branes (Choice~\ref{choice. ordering on objects}), any time we discuss a Lagrangian brane here, we assume that it equals $L^{(w)}$ for some $w \in \ZZ_{\geq 0}$ and some $L \in \LL_b$ (for some $b \in B$). In particular, if two Lagrangians are in the same fiber $E_b$, they are either equal or transverse. This is not strictly necessary, but it will make certain things easier. See also Warning~\ref{warning. transversality of branes}.
\end{remark}

\begin{notation}[$\vec L$]\label{notation. vec L}
Recall we have fixed a Liouville bundle $E \to B$.
Fix $d \geq 0$ and a smooth map $h: |\Delta^d| \to B$. We denote by
	\eqnn
	\vec L = (L_0,\ldots, L_d)
	\eqnd
an ordered $(d+1)$-tuple of branes with $L_i \subset h^*E$ contained above the $i$th vertex (Definition~\ref{defn. ith vertex})  of $|\Delta^d|$.
\end{notation}

\begin{remark}
In later notation, $h$ will play the role of the composite $\beta \circ j$. (See Definition~\ref{defn. mu^d counting} and Remark~\ref{remark. depends only on j o beta}.)
\end{remark}

Consider an ordered $(d+1)$-tuple $\vec L$ (Notation~\ref{notation. vec L}).

Because $|\Delta^d|$ is a smooth manifold with corners, one can construct a global 1-form $\Theta_{\vec L} \in \Omega^=(h^*E; \RR)$ realizing a Liouville structure on each fiber of $h^*E$. (Remark~\ref{remark. can choose liouville form}.)
Using the natural system maps $\nu_{\id}: \overline{\cS}_{d+1}^\circ \to |\Delta^d|$, we may also choose almost-complex structures $\JJ$ on $\nu_{\id}^* h^* E$ suitable for counting sections. (Definition~\ref{defn. suitable for sections}.)

We now specify the choices we make to guarantee that the moduli spaces of holomorphic sections $\cS_r \subset \overline{\cS^{\circ}_{d+1}} \xra{\nu} |\Delta^d| \to h^*E$ are well-behaved moduli spaces.

\begin{choice}[Liouville forms and almost complex structures suitable for counting sections]
\label{choice. floer data}

We begin with $d=0$. Note that $\Theta_{\vec L}$ is simply a choice of Liouville structure on the fiber of $E \to B$ determined by $h$, and likewise for $\JJ_{\vec L}$.\footnote{When $d=0$, a choice of $\JJ_{\vec L}$ suitable for counting sections is simply a choice of almost-complex structure $J$ on $E_b = E$ for which $J$ is compatible with the symplectic form and eventually cylindrical (Definition~\ref{defn. J conical near infinity}).} Given this choice of $\Theta_{\vec L}$, we may assume that each brane $L_i$ admits a primitive $f_i: L_i \to \RR$ such that $df_i = \Theta_{\vec L}|_{L_i}$ and such that $f_i$ has compact support for all $i=0, \ldots, d$. (In other words, we require that $[f^*\Theta] = 0$ in $H^1_c(L;\RR)$). (Any brane admits a deformation so that this holds.)

We proceed inductively on $d$. Assume that for all $d' < d$, for all $h': |\Delta^{d'}| \to B$, and for all $(d'+1)$-tuples $\vec L' = (L_0',\ldots, L_{d'})$, we have chosen $(\Theta_{\vec L'}, \JJ_{\vec L'})$ on $(h')^*E$.

Fix an ordered $(d+1)$-tuple $\vec L$. We choose $(\Theta_{\vec L}, \JJ_{\vec L})$ on $h^*E$ subject to the following conditions:

    \newenvironment{floer-data}{
    	  \renewcommand*{\theenumi}{($\Theta$\arabic{enumi})}
    	  \renewcommand*{\labelenumi}{($\Theta$\arabic{enumi})}
    	  \enumerate
    	}{
    	  \endenumerate
    }

    \begin{floer-data}
		\item\label{item. floer data trivial} (Constancy implies constancy.) If $h$ is constant, then so is $(\Theta_{\vec L},\JJ_{\vec L})$. More concretely, if $h: |\Delta^d| \to B$ is constant, then for a constant map $p: |\Delta^d| \to v$ to some (hence every) point of $|\Delta^d|$, we have that $\Theta_{\vec L} = p^* \Theta_{\vec L}|_{(h^*E)_v}$ and $\JJ_{\vec L} = p^* (\JJ_{\vec L})|_{(h^*E)_v}$.
    	\item \label{item. floer data inductive} (Inductive step.) If $\vec L' \subset \vec L$ is an order-preserving inclusion, consider the induced map $|\Delta^{d'}| \to |\Delta^{d}|$. Then the data $(\Theta_{\vec L'}, \JJ_{\vec L'})$ is equal to the pullback of $(\Theta_{\vec L}, \JJ_{\vec L})$
    along this induced map.
		\item\label{item. floer data smooth} (Smooth collaring.) Recall the collaring choices from Choice~\ref{choice. simplicial collars}. Suppose that $F \subset |\Delta^d|$ is a codimension one face. $F$ in particular determines an ordered $d$-tuple $\vec L' \subset \vec L$. We demand
			\eqnn
			\pi_F^* \Theta_{\vec L'} = \Theta_{\vec L}|_{U_F},
			\qquad
			\pi_F^* \JJ_{\vec L'} = (\JJ_{\vec L})|_{U_F}
			\eqnd
		where $\pi_F$ and $U_F$ are as in Choice~\ref{choice. simplicial collars}.
		\item\label{item. floer data transverse} (Transversality.) For any $0 \leq i < j \leq d$, let $\Pi_{ij}$ be the parallel transport along the simplicial edge from the $i$th vertex of $|\Delta^d|$ to the $j$th. (See Definition~\ref{defn. theta connection}.) We demand that $\Pi_{ij}(L_i)$ and $L_j$ are transverse.\footnote{See also Warning~\ref{warning. transversality of branes} regarding compatibility with~\ref{item. floer data trivial}.}
    	\item\label{item. floer data regular}  (Regularity.) Further, we demand that the associated linearized del-bar operators are regular, so that the holomorphic disk moduli spaces (see Definition~\ref{defn. mu^d counting}) are smooth manifolds.
		\item\label{item. defining function holomorphic} (Coherent barriers) Finally, we demand that there exists some neighborhood of $\del h^*E$ such that, with respect to some global function $\pi: \nbhd(\del h^*E) \to \CC_{\mathrm{Re } \geq 0}$ as in Remark~\ref{remark. barrier in families}, $\pi$ is $(\JJ_{\vec L})|_{VTE}$-holomorphic.
	\end{floer-data}
\end{choice}

\begin{remark}
We may now further motivate the collaring choices made for simplices in Choice~\ref{choice. simplicial collars}. If one chooses the above $\Theta_{\vec L}$ without collaring conditions, there is no guarantee that
the $\Theta_{\vec L}$ glue smoothly along faces of a simplex.
\end{remark}

\begin{warning}\label{warning. transversality of branes}
The reader may be irked by an apparent incompatibility between \ref{item. floer data trivial} and \ref{item. floer data transverse}. As stated, it is impossible to satisfy both conditions unless the branes $L_0,\ldots,L_d$ are a priori assumed transversal. This is the reason for Choice~\ref{choice. ordering on objects}; see also Remark~\ref{remark. transversality choices}.
\end{warning}

\begin{remark}
Recall we have fixed a natural system (Choice~\ref{choice. natural system}).
We may pull back our choices $(\Theta_{\vec L},\JJ_{\vec L})$ along the map $\cS_r \subset \overline{\cS}_{d+1}^{\circ} \xra{\nu_{\beta}} |\Delta^d|$. Then by \ref{NS. strip condition u}, along the strip-like ends, all our choices are translation-invariant. (Here, translation is by $[0,\infty)$ or by $(-\infty,0]$ as parametrized by the strip-like end.)
\end{remark}

\begin{remark}
We have used the notion of pulling back $\JJ_{\vec L}$---a choice of almost-complex structure suitable for counting sections (Definition~\ref{defn. suitable for sections})---in articulating the conditions of Choice~\ref{choice. floer data}. We note that this pullback is defined by utilizing the natural systems from Choice~\ref{choice. natural system}.
\end{remark}

\begin{prop}\label{prop:choices-exist}
There  exist choices $\{(\Theta_{\vec L}, \JJ_{\vec L})\}_{\vec L}$  satisfying all the conditions in Choice~\ref{choice. floer data}.
\end{prop}

\begin{proof}
This follows from a standard argument using induction; see for example Lemma~3.8 and Section~3.3 of~\cite{savelyev}. Perhaps the main point to note in our present work is how to choose the $\Theta_{\cL}$ compatibly. Given the bundle $h^*E \to |\Delta^d|$, the space of 1-forms $\Theta$ on $h^*E$ for which the fiberwise restrictions are Liouville forms is a smoothly contractible space (for example, it is easy to see that the space is convex). This contractibility is a necessary ingredient in the inductive step, as one must extend $\Theta$ from the boundary of an $n$-simplex to its interior.

When one must also account for brane structures (such as a trivialization of $\det^2_{\CC}(TM)$---see Remark~\ref{remark. liouville bundles different structure group}), the fact that the structure group has been reducted to $\Liouaut$ as opposed to $\Liouauto$ allows for this inductive step: By construction, the relevant trivialization from the boundary of an $n$-simplex extend to its interior.
\end{proof}

\subsection{A non-wrapped Fukaya category over a simplex}
Fix a smooth map $j: |\Delta_e^n| \to B$. (Note $|\Delta_e^n|$ is an extended simplex as in Notation~\ref{notation. extended simplices}.)

We define in this section the non-wrapped, directed Fukaya category $\cO_j$ associated to $j$.
The definition is inductive on $n$---we first define $\cO_j$ for all $j$ having domain of dimension $\leq n$, then for those $j$ with domain having dimension $n+1$.

\begin{remark}
The reader will note that $\cO_j$ only depends on the restriction of $j$ to the standard simplex
$|\Delta^n| \subset |\Delta_e^n|$. The reason we insist on the domain of $j$ being the extended simplex
$|\Delta_e^n|$ is to make use of homotopy-theoretic results concerning diffeological spaces;
the technical reasons for this will not arise prominently in this paper,
so we refer the reader to~\cite{oh-tanaka-smooth-approximation}.
\end{remark}

\begin{notation}[$b_i$ and $\LL_{b_i}$]
For every $0 \leq i \leq n$, let $b_i$ be the image of the $i$th vertex (Definition~\ref{defn. ith vertex}) of $|\Delta_e^n|$ under $j$.

Recall we have chosen a countable collection of branes and an order on these (Choice~\ref{choice. ordering on objects}). In particular, $\LL_{b_i}$ denote the countable collection of branes associated to $b_i$.
\end{notation}

\begin{defn}[Objects]\label{defn. obj of O}
An object of $\cO_j$ is a pair $(i,L)$ where
\begin{itemize}
\item $i \in \{0, \ldots, n\}$ and
\item $L \in \LL_{b_i}$.
\end{itemize}
To emphasize the role of the ordering $w$ we have chosen, we will often write an object as a triple
    \eqnn
    (i,L,w)
    \eqnd
where $w = w(L)$. (See Choice~\ref{choice. ordering on objects}.) We will also write this same object as
    \eqnn
    L_i^{(w_i)}
    \eqnd
from time to time.
\end{defn}

\begin{notation}[Parallel transport $\Pi$]\label{notation. parallel transport}
Fix a pair of objects $(L_0,i_0,w_0)$ and $(L_1,i_1,w_1)$. The integers $i_0$ and $i_1$ define a simplicial map $\beta: |\Delta^1| \to |\Delta^n| \subset |\Delta_e^n|$ sending the initial vertex of $|\Delta^1|$ to $i_0$ and the final vertex to $i_1$.

We let $h = j \circ \beta$. One also has an underlying ordered pair of branes $\vec L = (L_0, L_1)$. (The notation here is to be consistent with Notation~\ref{notation. vec L}.)

Because we have chosen $\Theta_{\vec L}$ for $h^* E$ (Choice~\ref{choice. floer data}), we have a parallel transport taking the initial fiber of $h^*E$ (i.e., the fiber above the initial vertex of $|\Delta^1|$) to the final fiber of $h^*E$.

We let $\Pi_{i_0, i_1}$ denote this parallel transport.
\end{notation}

We will render $\cO_j$ to be {\em directed} in the $w$ index; this means that the morphism complex from $(i, L, w)$ to $(i', L', w')$ will be zero unless $w < w'$, or $(i, L, w) = (i',L',w')$ (in which case the morphism complex is just the ground ring $R$ in degree 0). Concretely:

\begin{defn}[Morphisms]\label{defn. hom groups}
Fix two objects $(i_0, L_0, w_0)$ and $(i_1, L_1, w_1)$ of $\cO_j$.

We define the graded abelian group
	\eqnn
	\hom_{\cO_j}( (i_0, L_0, w_0) , (i_1, L_1, w_1))
	\eqnd
to be
	\eqnn
    \begin{cases}
    \bigoplus_{x \in \Pi_{i_0, i_1}(L_0^{(w_0)}) \cap L_1^{(w_1)}} \mathfrak{o}_{x}[-|x|]. & w_0 < w_1 \\
    R & (i_0, L_0, w_0) = (i_1, L_1, w_1) \\
    0 & \text{otherwise}.
    \end{cases}
    \eqnd
Here, $\Pi_{i_0,i_1}$ is the parallel transport map (Notation~\ref{notation. parallel transport}).
We also note that $\mathfrak{o}_{x}$ is the orientation $R$-module of rank one associated to the intersection point $x$, and $|x|$ is the Maslov index associated to the brane data.
\end{defn}

\begin{remark}
The set $x \in \Pi_{i_0, i_1}(L_0) \cap L_1$ is also in bijection with the set of flat sections of $h^*E \to |\Delta^1|$ (with respect to $\Theta_{(L_0,L_1)}$) beginning at $L_0^{(w_0)}$ and ending at $L_1^{(w_1)}$. (See Notation~\ref{notation. parallel transport}.)
\end{remark}

Now we define the operation $\mu^d$ for $d \geq 1$.

\begin{defn}[$\mu^d$ for the non-wrapped categories]\label{defn. mu^d counting}
As usual, fix a smooth map $j: |\Delta_e^n| \to B$.
For $d \geq 1$, fix a collection
	\eqnn
	\vec{L} = \{(i_0, L_0, w_0) , \ldots, (i_d, L_d, w_d)\}.
	\eqnd
We may assume $w_0 < \ldots < w_d$ by Definition~\ref{defn. hom groups} (otherwise $\mu^d$ is forced to be 0) .

Note that the integers $i_0, \ldots, i_d$ induce a simplicial map $\beta: |\Delta^d| \to |\Delta^n| \subset |\Delta_e^n|$ by sending the $a$th vertex of $|\Delta^d|$ to the $i_a$th vertex of $|\Delta^n|$. (This assignment, of course, need not be order-preserving.) Recall the map $\nu_{\beta}: \overline{\cS}_{d+1}^\circ \to |\Delta^n|$ as in Choice~\ref{choice. natural system}.

For a given collection of intersection points
	\eqnn
    x_a \in \Pi_{i_{a-1}, i_{a}}\left(L_{a-1}^{(w_{a-1})}\right) \cap L_{a}^{(w_{a})}
    \qquad
    (a = 1, \ldots, d)
	\eqnd
and
	\eqnn
    x_0 \in \Pi_{i_0, i_d}\left(L_0^{(w_{0})}\right) \cap L_d^{(w_{d})},
    \eqnd
we define
	\eqn\label{eqn. cM moduli of sections in O_j}
	\cM(x_d, \ldots, x_1; x_0)
	\eqnd
	to be the moduli space of holomorphic sections $u$
	\eqnn
		\xymatrix{
			& & & & E \ar[d] \\
		\cS_r \ar[r]_{\subset} \ar[urrrr]^{u} & \overline{\cS}^{\circ}_{d+1} \ar[r]_{\nu_\beta} & |\Delta^{d}| \ar[r]_-{\beta} & |\Delta^n| \subset |\Delta_e^n| \ar[r]_-j & B
		}
    \eqnd
satisfying the following boundary conditions:
	\enum
    \item Along the strip-line end near the $a$th puncture of $S$, $u$ converges to the parallel transport chord from $L_{a-1}^{(w_{a-1})}$ to $L_a^{(w_{a})}$ determined by $x_a$.
    \item Along the $a$th boundary arc of $S$, but outside the strip-like ends, $u$ is contained in the Lagrangian $L_a^{(w_a)} \subset E_{b_{i_a}}$. Note this makes sense due to the canonical trivialization of $E|_{\text{arc}} \cong E|_{b_{i_a}} \times \text{arc}$; this is a consequence of~\ref{NS. vertex condition u}.
    \enumd
As usual, the brane structures on the $L^{(w)}$ allow us to orient these moduli spaces, and predict their dimension based on the degrees of the $x_a$. We define
	\eqnn
    \mu^d(x_d, \ldots, x_1)
    =
    \sum_{x_0} \#{\cM(x_d,\ldots,x_1;x_0)} x_0
    \eqnd
where the number $\# \cM$ is counted with sign. In case our branes are not $\ZZ$-graded, we as usual declare the $x_0$ coefficient of $\mu^d$ to be zero when there is no zero-dimensional component of  $\cM(x_d,\ldots,x_1;x_0)$.
\end{defn}

\begin{remark}\label{remark. depends only on j o beta}
Given an ordered $(d+1)$-tuple of objects in $\cO_j$ with underlying branes $\vec L$, consider the induced map $\beta: |\Delta^d| \to |\Delta^n| \subset |\Delta_e^n|$. The $A_\infty$-operations are defined by moduli spaces depending only on $h = j \circ \beta$. (This follows from Definition~\ref{defn. mu^d counting} and ~\ref{item. floer data trivial}, \ref{item. floer data inductive}. Note that $h$ is the same $h$ as in Notation~\ref{notation. vec L}.)
\end{remark}

\begin{defn}\label{defn. cO_j}
Fix $j: |\Delta_e^n| \to B$. We let $\cO_j$ denote the $A_\infty$-category where
\begin{itemize}
	\item an object is the data of a brane $L^{(w)}$ in one of the vertex-fibers (as in Definition~\ref{defn. obj of O}),
    \item $\hom_{\cO_j}(L_0^{(w_0)}, L_1^{(w_1)})$ is as in Definition~\ref{defn. hom groups},
    \item The operations $\mu^d$ are as in Definition~\ref{defn. mu^d counting}.
\end{itemize}
\end{defn}

\begin{remark}
When a $\mu^d$ operation involves an element of an endomorphism hom-complex $\hom_{\cO_j}(L,L) = R$, the operation is fully determined by demanding that the unit of the ring $R$ be a strict unit of the $A_\infty$-category.
\end{remark}

\subsection{Proof of Theorem~\ref{theorem. O is Aoo}}

\begin{proof}
Because we have already set up the painstaking details, the theorem will be a standard consequence of (i) regularity, (ii) Gromov compactness for holomorphic sections, and (iii)  verifying that 1-dimensional moduli compactify in the usual way.

(i) Conditions~\ref{item. floer data transverse} and~\ref{item. floer data regular}
guarantee that our moduli are manifolds.

(ii) We established in Section~\ref{section. compactness  estimates} the estimates required for the standard Gromov compactness results for holomorphic sections in the Liouville setting.

(iii) Finally, we note that condition~\ref{item. floer data inductive}, together with
\ref{NS. star compatible} and \ref{NS. j functorial}, allow us to compactify
the $d$-ary moduli space using products of $d'$-ary moduli for $d' < d$.

In particular, thanks to the operadic comparability conditions spelled out
in \ref{NS. strip condition u} - \ref{NS. j functorial}, the usual Gromov-Floer compactification give rise to the following:

\begin{prop}
Let ${\bf x} = (x_d,\ldots,x_1)$ and consider the moduli space $\cM({\bf x};x_0)$ from \eqref{eqn. cM moduli of sections in O_j}.
 Then $\cM({\bf x};x_0)$ admits a compactification $\overline{\cM}({\bf x};x_0)$ whose boundary $\partial \overline{\cM}({\bf x};x_0) = \overline{\cM}({\bf x};x_0) \setminus \cM({\bf x};x_0)$ is naturally identified with the union
$$
\bigcup_{
\bar x \in L_{i+j}^{(w_{i+j})} \cap L_i^{(w_i)}} \overline{\cM}({\bf x}^1;x_0) \times
\overline{\cM}({\bf x}^2; \bar x).
$$
Here, $0 \leq i \leq d-j$ and
$$
{\bf x}^1=(x_d,\dots, x_{i+j+1}, \bar x, x_i,\dots, x_1),\quad
{\bf x}^2=(x_{i+j},\dots,x_{i+1}).
$$
\end{prop}
As usual, this description of $\partial \overline{\cM}({\bf x};x_0)$---applied to the case of $\dim \overline{\cM}({\bf x};x_0) = 1$---guarantees that the $A_\infty$ relations hold.

To finish the proof of Theorem~\ref{theorem. O is Aoo}, we must only prove that an injective simplicial map $\iota: |\Delta^n_e| \to |\Delta^{n'}_{e'}|$ induces a fully faithful functor $\cO_j \to \cO_{j'}$.

We define the functor to send an object $(i,L,w)$ to the object $(\iota(i), L, w)$. Because the morphism complexes from $(i,L,w)$ to $(i',L',w')$ depends only on the composite map
    \eqnn
    |\Delta^1| \to |\Delta^n| \to |\Delta^n_e| \to B
    \eqnd
(the first arrow  is the simplicial map sending the initial vertex of $|\Delta^1$ to $i$, and the terminal vertex to $i'$), that $\iota \circ j' = j$ means the morphism complexes admit a natural isomorphism
    \eqnn
    \hom_{\cO_j}( (i,L,w) , (i',L',w'))
    \xra{\cong}
    \hom_{\cO_{j'}}( (\iota(i),L,w), (\iota(i'),L',w')).
    \eqnd
Finally, Property~\ref{item. floer data inductive} guarantees that the moduli spaces of holomorphic sections defining the $A_\infty$ operations are also in natural bijection. (See Remark~\ref{remark. depends only on j o beta}.) Thus the functor $\cO_j \to \cO_{j'}$ is fully faithful.
\end{proof}

\begin{example}
Suppose $j: |\Delta_e^0| = |\Delta^0| \to B$ is the data of a point of $b \in B$. If the ordering function $w$ of Choice~\ref{choice. ordering on objects} is chosen to yield cofinal sequences of non-negative wrappings (see Remark~ref{remark. cofinal sequences})  $\cO_j$ is equivalent to the non-wrapped category $\cO$ that~\cite{gps} associates to the fiber $E_b$ above $b$. This is because the base case of $n=0$ in Choice~\ref{choice. floer data} implies that the boundary conditions for the holomorphic sections $u$ reduce to strip-like ends converging to intersection points $L_{a-1} \cap L_a$ in $M$. (When the pull-back bundle is canonically trivialized as $E_b \times |\Delta^d|$, sections are equivalent to maps $u: S \to E_b$.)
\end{example}

\section{Continuation maps}

Let $M$ be a Liouville manifold (or sector). If $L_0$ is a compact brane, any Hamiltonian isotopy from $L_0$ to $L_1$ induces an element in Floer cohomology $HF^*(L_0,L_1)$; this element is usually referred to as the continuation map, or sometimes the continuation element, associated to the isotopy.

Suppose $L_0$ is now a brane in a Liouville manifold (or sector). If $L_0$ is not compact and the Hamiltonian isotopy is not compactly supported, one must further impose the restriction that the isotopy be {\em non-negative} to construct the continuation map (Definition~\ref{defn. non negative wrapping}). Non-negativity yields the necessary $C^0$ and energy bounds to achieve Gromov compactness for moduli of disks (and continuation maps are constructed by counting holomorphic disks); see Theorem~\ref{thm:C0estimate+CL}.

Throughout, we fix a Liouville manifold (or sector) $M$ along with an exact Lagrangian isotopy of eventually conical branes
	\eqn\label{eqn. cL isotopy}
	\cL: [0,1] \times L_0 \to M
	\eqnd
in $M$. We assume $\cL$ is non-negative, and we review two constructions of continuation elements associated to $\cL$.

\subsection{Using once-punctured disks}
\label{sec:moving}

We assume that
\enum
	\item this isotopy is non-negative,
	\item For each $s$, the image of the time $s$ embedding $L_s$ is an (eventually conical) brane (Definition~\ref{defn. branes}), and
	\item $L_0$ is transverse to $L_1$.
\enumd

\begin{choice}[Choices for defining continuation map]\label{choice. chi_disk}
Choose a marked point $z_0 \in \del D^2$ and consider the Riemann surface with boundary $D^2 \setminus \{z_0\}$.
We equip $D^2 \setminus \{z_0\}$ with a strip-like end near $z_0$. Further, we choose a function
	\eqnn
	\chi: \del D^2 \to [0,1]
	\eqnd
such that $\chi$ is weakly increasing (with respect to the boundary orientation on $\del D^2$), is locally constant outside a compact set, and is onto.
\end{choice}

\begin{remark}\label{remark. u converges to constant path}
The non-negativity of the isotopy $\cL$ guarantees the usual $C^0$-estimates.\footnote{That is, the images of all the $v$ are contained in an a-priori-determined compact subset of $M$. See Section~\ref{section. compactness  estimates}.} Thus the finite energy condition implies that as $z \to z_0$, the map $u$ converges exponentially (with respect to the strip-like coordinates near $z_0$) to a constant path supported at an intersection point
$x \in L_0 \cap L_1$.
\end{remark}

\begin{notation}
Remark~\ref{remark. u converges to constant path} enables us to define the evaluation map
    $$
    \ev_{z_0}: \cM(D^2 \setminus \{z_0\} ; \cL^\chi) \to L_0 \cap L_1.
    $$
We denote
	\eqn\label{eq:cMDz0Lx}
    \cM(D^2 \setminus \{z_0\} ; \cL^\chi,x):= \ev_{z_0}^{-1}(x)
	\eqnd
for each $x \in L_0 \cap L_1$.
\end{notation}

\begin{prop}[Theorem C.3.1 \cite{oh:book2}]\label{prop. dimension of continuation moduli}
For a generic choice of isotopy $\cL$, $\cM(D^2 \setminus \{z_0\} ; \cL^\chi, x)$
is a smooth manifold of dimension given by
	$$
\dim \cM(D^2 \setminus \{z_0\} ; \cL^\chi, x) = \frac{n}{2} - \mu_{\cL}(x)
	$$
where $\mu_{\cL}(x)$ is the Maslov index of $x$ relative to $\cL$.
\end{prop}

\begin{remark}
The dimension count of Proposition~\ref{prop. dimension of continuation moduli} is compatible with the grading on $CF^*(L_0,L_1)$, in the sense that
	\eqn\label{eq:degreex}
|x| = \dim \cM(D^2 \setminus \{z_0\} ; \cL^\chi) = \frac{n}{2} - \mu_{\cL}(x).
	\eqnd
\end{remark}

\begin{remark} This definition of the \emph{cohomological} degree is
adopted because we put the output at $-\infty$ in the definition of
the Floer moduli space. Another choice would be to take $|x|$ to be
the codimension instead of the dimension of the relevant moduli space if the output were
put at $\infty$.
\end{remark}

\begin{construction}\label{construction. continuation cochain}
We define a Floer cochain
	\eqn\label{eq:c(cL)}
c^{\chi}_{\cL}: = \sum_{x \in L_0 \cap L_1; |x| = 0} n_\cL^\chi(x) \langle x \rangle
	\eqnd
where $n_\cL^\chi(x) = \#(\cM(D^2 \setminus \{z_0\} ; \cL^\chi))$ (counted with sign as usual using orientations).
\end{construction}

\begin{prop}\label{prop. continuation class using disks}
The cochain $c^{\chi}_{\cL}$ is a cocycle. Moreover, its Floer cohomology class $[c^{\chi}_{\cL}] \in HF^0(L_0,L_1)$
is independent of the choice of $\chi$.
\end{prop}

\begin{proof}[Proof of Proposition~\ref{prop. continuation class using disks}.]
We compute the matrix coefficient of the Floer coboundary $\mu^1 (c^{\chi}_{\cL})$. For each given $y \in L_0 \cap L_1$ with $|y| = 1$, we compute
its coefficient in the linear expression of $\mu^1 (c^{\chi}_{\cL})$
\begin{eqnarray*}
\langle \mu^1 (c^{\chi}_{\cL}), y \rangle & = &
\sum_{x \in L_0 \cap L_1}n_\cL^\chi(x) \langle \mu^1(\langle x \rangle),\langle y \rangle \rangle\\
& = & \sum_{x \in L_0 \cap L_1}\# \cM(y,x)\,  n_\cL^\chi(x).
\end{eqnarray*}

Here---by the standard compactness-and-gluing theorems---the last sum is nothing but the count of the boundary elements
of compact one-dimensional manifold $\overline \cM(D^2 \setminus \{z_0\};\cL,y)$; it hence vanishes.
This proves $\langle \mu^1 (c^{\chi}_{\cL}), y \rangle = 0$ for all $y$ with $|y| = 1$ and
so $\mu^1(c^{\chi}_{\cL}) = 0$. Therefore $c^{\chi}_{\cL}$ defines a
Floer cohomology class in $HF^0(L_0,L_1)$.

Now the standard compactness-cobordism argument proves the second statement
noting that the space of elongation functions $\chi$ is contractible (and in particular, connected).
\end{proof}

\begin{defn}[The continuation element]\label{defn. continuation element}
Let $\cL: [0,1] \times L_0 \to M$ be a non-negative, exact Lagrangian isotopy from $L_0$ to $L_1$.
We denote by
	\eqnn
	c_{\cL} \in HF^0(L_0,L_1)
	\eqnd
the cohomology class associated to the cochain in Construction~\ref{construction. continuation cochain}. We call it the {\em continuation element} associated to the isotopy.
\end{defn}

\subsection{Using strips}

\begin{choice}[$\rho$]\label{choice. elongation rho}
We fix an elongation function $\rho:\RR \to [0,1]$ given by
\begin{eqnarray}\label{eq:rho}
\rho(\tau) &= & \begin{cases} 1 \quad & \text{for } \, \tau \geq 1 \\
0 \quad & \text{for } \, \tau \leq 0
\end{cases}
\nonumber \\
\rho'(\tau) & > & 0 \quad \text{for } \, 0 < \tau < 1.
\end{eqnarray}
\end{choice}

\begin{remark}\label{remark. elongation function space is contractible}
For our purposes, any weakly monotone $\rho$ with value 0 near $-\infty$ and 1 near $\infty$ will suffice; we note that the space of such $\rho$ is contractible.
\end{remark}

\begin{notation}\label{notation. L^rho}
Given an exact Lagrangian isotopy $\cL$ and an elongation function $\rho$ as in Choice~\ref{choice. elongation rho}, we denote by
	$$
\cL^\rho: \tau \mapsto L_{\rho(\tau)}.
	$$
the induced $\RR$-parametrized isotopy.
\end{notation}

\begin{choice}\label{choice. J_(s,t)}
We also choose a smooth, 2-parameter family of eventually conical almost-complex structures on $M$ (Definition~\ref{defn. J conical near infinity})
	\eqnn
	[0,1] \times [0,1] \to \{
	\text{Eventually conical $J$}
	\},
	\qquad
	(s,t) \mapsto J_{(s,t)}.
	\eqnd
\end{choice}

\begin{construction}[Construction using strips]\label{construction. continuation strip}
Fix an exact Lagrangian isotopy $\cL$ and a brane $K$ such that $K \pitchfork L_i$ for $i = 0, \, 1$. The Floer continuation map
    \eqnn\label{eq:hCL}
    h_\cL^\rho:CF(K,L_0)\rightarrow CF(K,L_1)
    \eqnd
is defined by counting isolated solutions of the following system:
    \eqn\label{eq:moving}
        \begin{cases}
        {\frac {\partial u}{\partial\tau}} + J_{(\rho(\tau),t)}
        {\frac {\partial u}{\partial t}} =0 \\
        u(\tau ,0)\in K,\;\; u(\tau ,1)\in L_{\rho(1-\tau)}.
        \end{cases}
    \eqnd
See Figure~\ref{figure. continuation strip}.
\end{construction}

A similar argument to the proof Proposition~\ref{prop. continuation class using disks} shows the following:

\begin{prop}\label{prop. continuation class using strips}
$h_{\cL}^{\rho}$ is a chain map. Moreover, the map on cohomology
	\eqnn
	[h_{\cL}] : HF(K,L_0) \to HF(K,L_1)
	\eqnd
is independent of the choice of elongation function $\rho$.
\end{prop}

\begin{remark}\label{remark. chi rho compatibility}
Recall that in our constructions of the continuation elements and the Floer continuation map
in the previous Subsections, we have used two different kind of elongation functions, $\chi$ and $\rho$ (Choices~\ref{choice. chi_disk} and~\ref{choice. elongation rho})
where the domain of $\chi$ is $\del D^2 \setminus \{z_0\}$ and the domain of $\rho$ is $\RR$.
Different choices of such functions define the same map in cohomology.
However, when we study compatibility between the strip and disk definitions of continuation maps,
which requires us to examine a family of moduli spaces, we will need to exhibit a compatibility between $\chi$ and $\rho$. We will again have some freedom in exhibiting this compatibility; see~\eqref{eqn.chi-rho-compatibility}.
\end{remark}

\subsection{Recollections on \texorpdfstring{$\overline{\cM_4}$}{M4}}

We start our proof with considering
the configuration space $\cM_4$ of four boundary marked points of the unit disc modulo the action of $PSL(2,\RR)$.
We denote an element thereof by an equivalence class $[{\bf z}]$ of the tuple
	$$
{\bf z} = (z_0,z_1,z_2,z_3).
	$$
It is easy to see that $\cM_4$ is diffeomorphic to the open unit interval and its canonical
compactification by stable curves, denoted by $\overline{\cM}_4$, is
obtained by adding two points on the boundary of the open interval. Each of these two points represents a singular disc with two irreducible components.

More specifically,
modulo the action of $PSL(2;\RR)$, we may assume $z_0 = - 1 \in \partial D^2$. Then we consider
a diffeomorphism $\mathfrak{r}: \cM_4 \to \RR_{> 0}$ given by the (real) cross ratio,
	\eqn\label{eq:crossratio}
\mathfrak{r}([1,z_1,z_2,z_3]) = \frac{w_2-w_3}{w_1-w_2}; \quad w_i = \log z_i
	\eqnd
where $w_i \in \del \HH \subset \CC$.
(Here we take the logarithm $w_i = \log z_i$ with respect to the branch cut along the positive real axis.)

\begin{notation}[$\varphi_r$]\label{notation.varphi_r}
For later use, for each $r \in (0,\infty)$ we denote by
	\eqnn
	\varphi_r: D^2 \setminus \{z_0,z_1\} \to \RR \times [0,1]
	\eqnd
the unique conformal map satisfying

	\eqn\label{eq:varphir}
\begin{cases}
	\varphi_r(z_0) = -\infty, \\
	\varphi_r(z_1) = \infty, \\
	\varphi_r(z_2) = (r,1),\\
	\varphi_r(z_3) = (-r,1)
\end{cases}
	\eqnd
for $r = \mathfrak{r}([z_0,z_1,z_2,z_3])$.
\end{notation}

This realization on $\RR \times [0,1]$ (of the unit disk's boundary points as prescribed by elements of $\cM_4$) will be important in the study of the continuation equation and its relationship with the
continuation element.

The two boundary points of $\overline{\cM}_4$ represent singular curves of the types
    \begin{eqnarray}\label{eq:r=infty}
    {\mathfrak{r}}^{-1}(0) & = & (D^2, (z_0,z_1,\zeta)) \# (D^2,(\zeta,z_2,z_3)), \nonumber \\
    {\mathfrak{r}}^{-1}(\infty)&  = & (D^2, (z_0,\zeta,z_3)  \# (D^2,(\zeta,z_1,z_2))).
    \end{eqnarray}
Under the above diffeomorphism $\mathfrak{r}: \overline \cM_4 \to [0, \infty]$,
the unique conformal map $\varphi_r: D^2 \setminus \{z_0,z_1\} \to \RR \times [0,1]$ respects
degenerations of the domain and of the target and we can express the limit
of the sequence $[z_0,z_1,z_2,z_3]$ in $\overline \cM^4$ at $\mathfrak{r} = \infty$ as a join of
the morphisms between two stable curves
\begin{eqnarray}
\psi: (D^2, (z_0,z_1,\zeta)) & \to & (\RR \times [0,1], \{(0,1)\}), \label{eq:psi}\\
\varphi: (D^2,\zeta) & \to & \Theta_- \label{eq:varphi}
\end{eqnarray}
where $\psi$ and $\varphi$ are uniquely defined by the symmetry \eqref{eq:varphir}
condition imposed on $\varphi_r$. We describe $\Theta_-$ in Notation~\ref{notation. theta minus} below.

\begin{figure}[ht]
    \eqnn
			\xy
			\xyimport(8,8)(0,0){\includegraphics[width=2in]{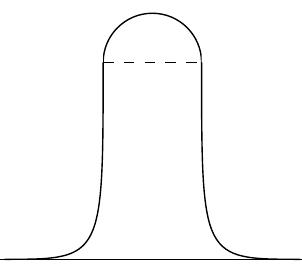}}
					,(4,8)*{\cL}
					,(4,5.8)*{y}
					,(5.8,3.4)*{L_0}
					,(2.3,3.4)*{L_1}
					,(7,0.8)*{x}
					,(1.1,0.8)*{x'}
					,(4,0.2)*{X}
			\endxy
    \eqnd
    \begin{figurelabel}\label{figure.disk-strip-glued}
    The glued image for the composition of $\mu^2(c_{\cL},-)$.
    \end{figurelabel}
\end{figure}

\begin{notation}[$\Theta_-$]\label{notation. theta minus}
We denote by $\Theta_-$ the domain (equipped with the strip-like coordinates)
of the relevant moduli spaces, and denote by $\Theta_- \# Z$ the nodal curve obtained by the obvious grafting. (See Figure \ref{figure.disk-strip-glued} for the image of the grafted domain.)
We mention that we have conformal equivalences $\Theta_- \cong D^2 \setminus \{z_0\}$ and
$Z \cong D^2 \setminus \{z_0,z_1, z_2\}$. We take the following explicit model for $\Theta_-$:
Consider the domain
$$
\{z \in \mathbb{C} \mid |z| \leq 1, \, \mathrm{Im } z \geq 0\} \cup \{z \in \mathbb{C} \mid |\mathrm{Re } z| \leq 1,
\, \mathrm{Im } z \leq 0\} \label{eq:Theta-}
$$
and take its smoothing around $\mathrm{Im } z = 0$ that keeps the reflection symmetry about the $y$-axis
of the domain. Then we take
\eqn
Z = \{z \in \mathbb{C} \mid 0 \leq \mathrm{Im } z \leq 1\} \setminus \{(0,1)\}. \label{eq:Z}
\eqnd
Again we equip $Z$ with a strip-like coordinate at $z = (0,1)$ that keeps the reflection symmetry.
\end{notation}

\begin{remark}
We are using $\Theta_-$ to realize the degeneration of curves in~\eqref{eq:r=infty} through
a Riemann surface that is not only conformally equivalent to $D^2 \setminus \{z_0\}$
but also respects the symmetry of the kind \eqref{eq:varphir}. $\Theta_-$ will be useful in studying the degeneration of Floer
moduli spaces involved in understanding how the Abouzaid functor $\cF$ in \cite{oh-tanaka-actions} interacts with continuation maps.
\end{remark}

\subsection{Proof of Theorem~\ref{thm:hcL=mu2ccL}}
\label{subsubsec:hcL=mu2ccL}

In this subsection we give the proof of Theorem~\ref{thm:hcL=mu2ccL}.
We start with Proposition \ref{prop. continuation class using strips}. Using the
independence of the map $[h_\cL^\rho]$ on $\rho$, we deform $\rho$ through a suitably chosen one-parameter
family $\{\rho_r\}_{0 < r \leq 1}$ starting with $\rho_1 = \rho$ whose construction is now in order.
We would like to degenerate the moduli space $\cM(K, \cL^{\rho_r};x_-,x_+)$ to
the one associated to the right-hand side of Theorem~\ref{thm:hcL=mu2ccL} as $r \to 0$.

Applying $\varphi_1$ (see Notation~\ref{notation.varphi_r}), let us study the collection of maps
	\eqn\label{eq:cMcLK}
	v: D^2 \setminus \{z_0, z_1\}  \to M
	\qquad
	\text{
		such that $u = v \circ \varphi_1^{-1}$ satisfies~\eqref{eq:moving}
	}
	\eqnd
satisfying the finite energy condition. Any such solution converges to $x_\pm$ with
$x_+ \in K \cap L_0$ and $x_- \in K \cap L_1$ as $\tau \to \pm \infty$ respectively
in the coordinate $(\tau,t) \in \RR \times [0,1]$.
Given points $x_-$ and $x_+$, we denote by
	\eqn\label{eqn. moduli moving boundary strips continuation}
	\cM(K, \cL^\chi;x_-,x_+).
	\eqnd
the moduli space of (equivalence classes of) pairs $(v; z_0,z_1)$,
where $v$ is a solution of \eqref{eq:cMcLK} converging to $x_-$ and $x_+$,
modulo the biholomorphisms of $D^2$. This is naturally isomorphic to $\cM(K, \cL^{\rho};x_-,x_+)$
if we set $\chi = \rho \circ \varphi$ by definition.
Therefore to make the following discussion
consistent with the compactification of $\cM_4$,
we will degenerate the moduli space \eqref{eqn. moduli moving boundary strips continuation}
instead by deforming $\chi$ used in Choice \ref{choice. chi_disk} through
a one-parameter family of $\chi$'s parameterized by $\cM_4 \cong (0, \infty)$ via the
diffeomorphism $\mathfrak{r}: \cM_4 \to \RR_{\geq 0}$ as follows.

We first fix an elongation function $\rho: \RR \to [0,1]$ and define
	\eqn\label{eqn.chi-rho-compatibility}
	\chi_r = \rho \circ \varphi_r
	\eqnd
for $r > 0$ and denote
$$
\rho_r = \rho \circ (\varphi_r \circ \varphi_1^{-1}) = \chi_r \circ  \varphi_1^{-1}.
$$
Note that $\rho_1 = \rho$. Then we introduce a parameterized moduli space
of $(v; {\bf z})$ with ${\bf z} = (z_0,z_1,z_2,z_3)$ by adding two more marked points
$z_2, \, z_3$ to $(u; z_0, z_1)$. We have the natural fibration
	$$
\mathfrak{ft}: \cM_4(K, \cL^{\chi_{{\mathfrak{r}}}};x_-,x_+) \to \cM_4
	$$
whose fiber is given by
$
\cM_4(K, \cL^{\chi_r};x_-,x_+)
$
with $\chi_r:= \rho \circ \varphi_r$ for $r = \mathfrak r([z_0,z_1,z_2,z_3])$: Each element of
$\cM_4(K, \cL^{\chi_r};x_-,x_+)$ is a pair
$$
(v; z_0, z_1, z_2, z_3) \quad \text{satisfying } \, \mathfrak{r}([z_0, z_1, z_2, z_3]) = r
$$
where $v$ is defined on $D^2 \setminus \{z_0, z_1\}$. Since adding two additional
(free) marked points $z_2,\, z_3$ increases the dimension by 2, we need to cut it down by
putting a codimension 1 constraint on the location of each of $v(z_2)$ and $v(z_3)$. We do this by
taking local codimension 1 slices transversal to the image of $v$ at $v(z_2)$ and $v(z_3)$,
respectively.
For this purpose, we use the following lemma: Recall that in the situation of
Theorem~\ref{thm:hcL=mu2ccL}
the moduli space $\cM_4(K, \cL^\chi; x_-,x_+)$ is zero dimensional and compact for $\chi = \chi_1$.
In particular it consists of finitely many elements. We enumerate them by
$$
\cM_4(K, \cL^\chi;x_-,x_+) = \{v^{(1)}, \ldots, ,v^{(N)}\}.
$$
\begin{lemma}\label{lem:faraway} Suppose the moduli space $\cM(K, \cL^{\chi};x_-,x_+)$
is nonempty and regular for $\chi = \chi_1$. Then we can choose marked points
${\bf z}^\ell = (z_0^\ell, z_1^\ell, z_2^\ell, \, z_3^\ell)$ so that
\begin{enumerate}
\item ${\mathfrak r}([z_0^\ell, z_1^\ell, z_2^\ell, \, z_3^\ell]) = 1$,
\item $v^{(\ell)}$ is immersed at $z^\ell_2$ and $z^\ell_3$ for all $\ell = 1, \ldots, N$.
\item there exists some $\delta_0 > 0$ such that
$$
d\left(v^{(\ell)}(z^\ell_i),v^{(\ell')}(z^{\ell'}_j)\right) \geq \delta_0
$$
for any pair $(\ell,i) \neq (\ell',j)$ with
    $j = 2, \, 3$ and $1 \leq \ell \leq N$.
\end{enumerate}
\end{lemma}

The proof of this lemma---a simple
application of the unique continuation of the image of pseudoholomorphic curves \cite{FHS,oh-JGA}
and the implicit function theorem---is omitted.

We then pick a local transversal slice $S_i^\ell \subset M$
for each $i = 2, \, 3$ and $1 \leq \ell \leq N$ such that
\begin{itemize}
\item $\{S^\ell_i\}$ are pairwise disjoint for all $\ell$ and $i=2, \,3$,
\item $v^{(\ell)}(z_i^\ell) \in S^\ell_i \cap L_{s^{\ell}_i}$ for all $\ell$ for each $i = 2, \, 3$,
\item Both $S^\ell_i \pitchfork v^{(\ell)}$ in $M$ and $(S^\ell_i \cap L_{s_i^\ell}) \pitchfork \del v^{(\ell)}$
hold at $z_i^{0}$ for $i = 2, \, 3$ where $s_i^\ell = \chi_{r^\ell}(z_i^\ell)$ with $r^\ell = \mathfrak{r}([z_0^\ell,z_1^\ell,z_2^\ell,z_3^\ell])$.
\end{itemize}
It follows from Condition 3 of Lemam \ref{lem:faraway} that by taking $S^\ell_i$ sufficiently small,
we may also assume
\eqn\label{eq:dS2S3}
 d(S^\ell_2, S^\ell_3) > \frac{\delta_0}{2}
\eqnd
for all $\ell = 1, \ldots, N$. We denote these collection of $S^\ell_i$ by
$$
\cS_i = \{S^1_i, \ldots, S^N_i\}, \, i = 2,\, 3.
$$
\begin{prop}\label{prop:slice} Suppose the moduli space $\cM(K, \cL^{\chi};x_-,x_+)$
is nonempty and transversal. Let ${\cS}^\ell_i$ for $i=2, \, 3$ be the collection of
the local slices for the $\cM(K, \cL^{\chi};x_-,x_+)$ chosen in
Lemma \ref{lem:faraway} above. We define a subset of $\cM_4(K, \cL^{\chi_{\mathfrak{r}}};x_-,x_+)$ by
\begin{eqnarray*}
&{}& \cM_4^{{\cS}_2,{\cS}_3}(K, \cL^{\chi_{\mathfrak{r}}};x_-,x_+) \\
& =& \left\{(r,[(v; {\bf z})]) \in
\cM_4(K, \cL^{\chi_{\mathfrak{r}}};x_-,x_+) \Big\vert v(z_i) \in \cup_{\ell=1}^N S^\ell_i, \, i=2,\,3 \right\} \bigcap \mathfrak{r}^{-1}((0,1])
\end{eqnarray*}
where ${\bf z} = (z_0,z_1,z_2,z_3)$. Then provided the isotopy $\{L_s\}$ is sufficiently small in fine $C^\infty$ topology,
this moduli space is a smooth submanifold of $\cM_4(K, \cL^{\chi_{\mathfrak{r}}};x_-,x_+) \to (0,\infty)$ of
codimension 2 and so of one dimension.
Moreover the same property also persists to the compactification
$$
\overline{\cM}_4^{{\cS}_2,{\cS}_3}(K, \cL^{\chi_{\mathfrak{r}}};x_-,x_+) \subset
\overline{\cM}_4(K, \cL^{\chi_{\mathfrak{r}}};x_-,x_+)\cap \mathfrak{r}^{-1}([0,1]):
$$
\begin{enumerate}
\item $v$ is immersed at $z_2$ and $z_3$ for all $(v; {\bf z}) \in \overline{\cM}(K, \cL^{\chi_{\mathfrak{r}}};x_-,x_+)$,
\item $v(z_i) \in \cup_{\ell=1}^N S^\ell_i, \, i=2,\, 3$.
\end{enumerate}
Both $S^\ell_i \pitchfork v$ in $M$ and $(S^\ell_i \cap L_{s_i}) \pitchfork \del v$
hold at $z_i$ for $i = 2, \, 3$ where $s_i = \chi_{r}(z_i)$ with $r = \mathfrak{r}([z_0,z_1,z_2,z_3])$
for all $\mathfrak{r}^{-1}([0,1])$.
\end{prop}
\begin{proof} This is a standard local normal slice theorem;
we refer readers to \cite[p.424]{fooo2}  for a proof and for the details of such a construction
at the interior marked points. The current case of boundary marked points is the same
except the following differences
\begin{itemize}
\item Here our slice $S^\ell_i$ is of codimension 1 in $M$ instead of codimension 2
in $M$,
\item We want this slice theorem for the whole family of the moduli spaces
$\overline{\cM}_4(K, \cL^{\chi_{\mathfrak{r}}};x_-,x_+) \to \mathfrak{r}^{-1}([0,1]) \subset \overline{\cM}^4$.
\end{itemize}
The $C^\infty$ smallness of the isotopy is required to ensure these properties for the whole family.
\end{proof}

The $C^\infty$ smallness required in the proposition can be always achieved by
breaking the given isotopy into a concatenation of smaller isotopies by choosing times
$$
0 = t_0 < t_1 < t_2 < \cdots < t_N = 1.
$$
We also remark that for the proof of Theorem~\ref{thm:hcL=mu2ccL}
will follow if we prove it for the portion on each interval $[t_i,t_{i+1}]$ of the given isotopy.
With these being said, we will assume from now on that the isotopy is sufficiently
$C^\infty$ small so that the global slices $S^\ell_2,\, S^\ell_3$ exist independently of $r \in [0,1]$.

Then by construction we have the following obvious one-one correspondence
$$
\cM_4^{{\cS}_2,{\cS}_3}(K, \cL^{\chi_r};x_-,x_+) \cong \cM(K, \cL^{\chi_r};x_-,x_+) \cap (\mathfrak{r})^{-1}((0,1])
$$
where we denote
$$
\cM_4^{{\cS}_2,{\cS}_3}(K, \cL^{\chi_r};x_-,x_+): = \cM_4^{{\cS}_2,{\cS}_3}(K, \cL^{\chi_{\mathfrak{r}}};x_-,x_+) \cap (\mathfrak{r})^{-1}(r).
$$
Therefore we observe that there is a canonical diffeomorphism
$$
\cM_4^{{\cS}_2,{\cS}_3}(K, \cL^{\chi_r};x_-,x_+) \cong \cM(K, \cL^{\rho_r};x_-,x_+)
$$
induced by $v \mapsto v \circ \varphi_r^{-1}$ for the mapping part.
We introduce a parameterized Floer continuation moduli space
	$$
\cM^{\text{para}}(\cL,L;x_-,x_+) = \coprod_{r \in (0,1]} \{r\} \times \cM(K, \cL^{\rho_r};x_-,x_+)
	$$
\emph{defined on the domain $\RR \times [0,1]$}.
We have a natural fiberwise isomorphism
	$$
\widetilde{\mathfrak r}:\cM_4^{{\cS}_2,{\cS}_3}(K, \cL^{\chi_{{\mathfrak{r}}}};x_-,x_+ )
\to \cM^{\text{para}}(K, \cL;x_-,x_+)
	$$
given by
$$
\widetilde{\mathfrak r}(v;z_0,z_1,z_2,z_3) = \left(v \circ \varphi_r^{-1}, \mathfrak{r}([z_0,z_1,z_2,z_3])\right)
$$
which makes the following commutative diagram commute,
	$$
\xymatrix{\cM_4^{{\cS}_2,{\cS}_3}(K, \cL^{\chi_{{\mathfrak{r}}}};x_-,x_+)
 \ar[d]^{\mathfrak{ft}}\ar[r]^{\widetilde{\mathfrak r}}
& \cM^{\text{para}}(K, \cL;x_-,x_+) \ar[d]\\
(\mathfrak r \circ \mathfrak{ft})^{-1}((0,1]) \ar[r]^{\mathfrak r} & (0,1]}
	$$
where $\mathfrak{ft}$ is the restriction of the forgetful map
$$
\mathfrak{ft}:\cM_4(K, \cL^{\chi_{{\mathfrak{r}}}};x_-,x_+) \to \cM^4.
$$
We note that the condition $v(z_i) \in S^\ell_i \cap L_{s_i}$ for $i = 2,\, 3$ in
Proposition \ref{prop:slice} and \eqref{eq:dS2S3} imply the separating condition  $d(v(z_2),v(z_3)) \geq \frac{\delta_0}{2} > 0$
for all $v \in {\overline{\cM}}^{{\cS}_2,{\cS}_3}(K, \cL^{\chi_{\mathfrak{r}}};x_-,x_+)$.
Then it follows from the standard Gromov-Floer compactification and
one-jet transversality that the compactified moduli space
$$
\overline{\cM}_4^{{\cS}_2,{\cS}_3}(K, \cL^{\chi_{{\mathfrak{r}}}};x_-,x_+)
$$
carries its boundary consisting of the types
\begin{enumerate}
\item $\cM_4^{{\cS}_2,{\cS}_3}(K, \cL^{\chi_{{\mathfrak{r}}}});x_-,x_+)|_{r=0} \cong
\cM_{3}^{{\cS}_2,{\cS}_3}(\cL^\chi;y)\# \cM_{3}(K,L_0,L_1; y,x_-, x_+)$
with $y \in L_0 \cap L_1$,
\item $\cM_4^{{\cS}_2,{\cS}_3}(K,\cL^{\chi_{{\mathfrak{r}}}};x_-,x_+)\Big|_{r=1}$,
\item\label{item. boundary of cm4} $\cM_1^{{\cS}_3}(K,L_1;x_-,x') \# \cM_3^{{\cS}_2}(K, \cL^{\mathfrak{r}};x',x_+)$
for some $x' \in K \cap L_1$ with $|x_-| = |x'| + 1$,
\item\label{item. another boundary of cm4} $\cM_3^{{\cS}_3}(K, \cL^{\chi_{{\mathfrak{r}}}};x_-,x)
\#\cM_1^{{\cS}_2}(K,L_0;x,x_+)$ for some
$x \in K \cap L_0$ with $|x| = |x_+| + 1$.
\end{enumerate}
Let us explain the notation.
\begin{itemize}
\item $\cM_3^{{\cS}_2,{\cS}_3}(\cL^{\chi};x)$ is the moduli space of equivalence classes of pairs
$(w;\zeta, z_2,z_3)$ where
  $w$ is a function $D^2 \setminus \{\zeta\} \to (M, \cL)$ satisfying
 $$
  w(z_2) \in \bigcup_{\ell=1}^N S^\ell_2, \, w(z_3) \in \bigcup_{\ell=1}^N S^\ell_3
 $$
and
	\eqnn
\begin{cases}
\delbar w = 0, \\
\lim_{z \to \zeta} w(z) = x, \\
w(z) \in L_{\chi(z)} & \text{for } \, z \in \del D^2 \setminus \{\zeta\}
\end{cases}
	\eqnd
for $\chi: =\rho \circ \varphi|_{\partial D^2 \setminus \{\zeta\}}$ where $\varphi$ is
the map given in \eqref{eq:varphi}. The forgetful map $(w;\zeta,z_2,z_3) \mapsto w$
induces an isomorphism between $\cM_3^{{\cS}_2,{\cS}_3}(\cL^{\chi};y)$ and the moduli space $\cM(D^2\setminus\{\zeta\};\cL^\chi,y)$ studied in
Section \ref{sec:moving}. Therefore the count of $\cM_3^{{\cS}_2,{\cS}_3}(\cL^{\chi};y)$
gives rise to the operation $c_{\cL}$.
\item $\cM_3(K,L_0,L_1;y,x_-,x_+)$ is the moduli space whose count encodes the obvious coefficients in the usual $\mu^2$ operation.
\item $\cM_1^{{\cS}_3}(K,L_1;x_-,x') \cong \cM(K,L_1;x_-,x')$ and
$\cM_1^{{\cS}_2}(K,L_0;x,x_+)\cong \cM(K,L_0;x,x_+)$ are the similarly defined moduli of strips (with non-moving boundary conditions) whose counts encode the obvious coefficients in the usual $\mu^1$ operation.
We note that
one-jet transversality is used to establish that the bubbling of such type
cannot occur at the interior parameter $0 < r < \infty$ by dimension counting.)
\end{itemize}

Now we define the map $\cH: CF^*(K,L_0) \to CF^{* - 1}(K,L_1)$ of degree $-1$ by
	$$
\cH(z) = \sum_{|x| = |z|-1} \#\left(\cM_4^{{\cS}_2,{\cS}_3}(K, \cL^{\mathfrak{r}};z,x)\right) \langle x\rangle.
	$$
We mention that $\dim \cM_4^{{\cS}_2,{\cS}_3}(K, \cL^{\mathfrak{r}};z,x) = 0$ since $|z| = |x|-1$.
By summing up the sign counts of all the points in
$\del \overline{\cM}_4^{{\cS}_2,{\cS}_3}(K, \cL^{\mathfrak{r}};x_-,x_+)$ and utilizing
\begin{itemize}
    \item The isomorphism
    $\cM_4^{{\cS}_2,{\cS}_3}(K, \cL^{\mathfrak{r}};x_-,x_+)\Big|_{r=1} \cong \cM(K, \cL^\rho;x_-,x_+)$
    since $\rho_1 = \rho$,
    \item The definitions of $\mu^1$ and $\mu^2$ (which are standard),
    \item The definition of $c^{\chi}_{\cL}$ (Construction~\ref{construction. continuation cochain}),
    \item The definition of the continuation map $h_\cL^\rho$ (Construction~\ref{construction. continuation strip}), and
    \item The definition of $\cH$ just given,
\end{itemize}
we have
proven the identity
	$$
h_\cL^\rho - \mu^2(c^{\chi}_{\cL},\ast) = \mu^1 \cH + \cH \mu^1.
	$$
This proves that the two maps $h_\cL^\rho, \, \mu^2(c^{\chi}_{\cL},\ast)$ are chain-homotopic.
By taking cohomology, we have finished the proof.

\begin{remark}
The scheme we use in the proof is in the same spirit as in the definition of the
stable map topology given in~\cite{fukaya-ono:arnold}. There, convergence
for a sequence of maps with \emph{unstable} domain is defined by first
stabilizing the domain by adding additional marked points, taking
the limit, then finally forgetting the added extra marked points.
(See also \cite[Section 9.5.2]{oh:book2} for an amplification of this trichotomy.)
In the above proof, we take a choice of the minimal (and so optimal) number
of additional marked points by taking suitable transversal normal slices for the
convergence proof: This is guided by our goal to prove the identity spelled out in Theorem~\ref{thm:hcL=mu2ccL}. All these steps are a part of a standard process
in the study of moduli space of pseudoholomorphic curves in general---for example,
in the construction of the Kuranishi structure and abstract perturbation
of the moduli space of stable maps.
\end{remark}

\bibliographystyle{amsalpha}
\bibliography{biblio}

\end{document}